\documentclass[a4paper,reqno]{amsart}
\usepackage{geometry}
\usepackage{mathrsfs}
\usepackage{syntonly}
\usepackage{amsmath}
\usepackage{amsthm}
\usepackage{amsfonts}
\usepackage{amssymb}
\usepackage{latexsym}
\usepackage{amscd,amssymb,amsopn,amsmath,amsthm,graphics,amsfonts,mathrsfs,accents,enumerate,verbatim,calc}
\usepackage[dvips]{graphicx}
\usepackage[colorlinks=true,linkcolor=blue,citecolor=blue]{hyperref}
\input xy
\xyoption{all}
\def\del{\delta}

\usepackage{color}

\numberwithin{equation}{section}

\theoremstyle{plain}

\newtheorem{thm}{Theorem}[section]
\newtheorem{cor}[thm]{Corollary}
\newtheorem{lem}[thm]{Lemma}
\newtheorem{prop}[thm]{Proposition}

\newtheorem{defn}[thm]{Definition}
\newtheorem{exm}[thm]{Example}

\newtheorem{rem}[thm]{Remark}

\newdir{ >}{{}*!/-5pt/\dir{>}}

\newcommand{\add}{\operatorname{add}\nolimits}
\newcommand{\Fac}{\operatorname{Fac}\nolimits}

\newcommand{\Hom}{\operatorname{Hom}\nolimits}

\newcommand{\End}{\operatorname{End}\nolimits}

\newcommand{\Ext}{\operatorname{Ext}\nolimits}

\newcommand{\pd}{\operatorname{pd}\nolimits}

\renewcommand{\mod}{\mathsf{mod}\hspace{.01in}}
\newcommand{\Cone}{\operatorname{Cone}\nolimits}
\newcommand{\CoCone}{\operatorname{CoCone}\nolimits}

\newcommand{\M}{\mathcal M}

\newcommand{\B}{\mathcal B}
\newcommand{\uB}{\underline{\B}}
\newcommand{\oB}{\overline{\B}}
\newcommand{\U}{\mathcal U}

\newcommand{\A}{\mathcal A}
\newcommand{\W}{\mathcal W}
\newcommand{\h}{\mathcal H}
\newcommand{\s}{\mathcal S}
\newcommand{\T}{\mathcal T}

\newcommand{\I}{\mathcal I}
\newcommand{\D}{\mathcal D}

\newcommand{\R}{\mathcal R}
\newcommand{\X}{\mathscr X}
\newcommand{\Y}{\mathcal Y}

\newcommand{\C}{\mathcal C}

\newcommand{\EE}{\mathbb E}

\newcommand{\svecv}[2]{\left(\begin{smallmatrix}
      #1 \\
      #2
    \end{smallmatrix}\right)}

\newcommand{\svech}[2]{\left(\begin{smallmatrix}
      #1 & #2
\end{smallmatrix}\right)}

\renewcommand{\emph}{\textit}
\renewcommand{\phi}{\varphi}

\begin{document}

\title{Relative rigid subcategories and $\tau$-tilting theory}\footnote{Yu Liu was supported by the Fundamental Research Funds for the Central Universities (Grant No. 2682018ZT25) and the National Natural Science Foundation of China (Grant No. 11901479). Panyue Zhou was supported by the National Natural Science Foundation of China (Grant Nos. 11901190 and 11671221), and by the Hunan Provincial Natural Science Foundation of China (Grant No. 2018JJ3205), and by the Scientific Research Fund of Hunan Provincial Education Department (Grant No. 19B239).}
\author{Yu Liu and Panyue Zhou}
\address{School of Mathematics, Southwest Jiaotong University, 610031, Chengdu, Sichuan, People's Republic of China}
\email{liuyu86@swjtu.edu.cn}
\address{College of Mathematics, Hunan Institute of Science and Technology, 414006, Yueyang, Hunan, People's Republic of China}
\email{panyuezhou@163.com}
\thanks{The authors would like to thank Professor Dong Yang and Professor Bin Zhu for helpful discussions.}
\begin{abstract}
Let $\B$ be an extriangulated category with enough projectives $\mathcal P$ and enough injectives $\mathcal I$, and let
 $\R$ be a contravariantly finite rigid subcategory of $\B$ which contains $\mathcal P$.  We have an abelian quotient category $\h/\R\subseteq \B/\R$ which is equivalent $\mod (\R/\mathcal P)$.  In this article, we find a one-to-one correspondence between support $\tau$-tilting (resp. $\tau$-rigid) subcategories of $\h/\R$ and maximal relative rigid (resp. relative rigid) subcategories of $\h$, and show that support tilting subcategories in $\h/\R$ is a special kind of support $\tau$-tilting subcategories. We also study the relation between tilting subcategories of $\B/\R$ and cluster tilting subcategories of $\B$ when $\R$ is cluster tilting.
\end{abstract}
\maketitle

\section{Introduction}

In mathematics, especially representation theory, classical tilting theory describes a way to relate the module categories of two algebras using so-called tilting modules and associated tilting functors. Classical tilting theory was motivated by the reflection functors introduced by Bernstein, Gelfand and  Ponomarev \cite{BGP}. These functors were reformulated by Auslander, Platzeck and Reiten \cite{APR}, and generalized by Brenner and Butler \cite{BB}.

Inspired by classical tilting theory, Buan, Marsh, Reineke, Reiten and Todorov \cite{BMRRT} introduced
cluster tilting objects in the context of cluster categories associated to a hereditary algebra, in order to categorify certain phenomena occurring in the theory
of cluster algebras developed by Fomin and Zelevinsky \cite{FZ}. Cluster categories have led to new developments in the theory of the canonical basis and its dual. They are providing insight into cluster algebras
and the related combinatorics, and have also been used to establish a new kind
of classical tilting theory, known as cluster tilting theory, which generalizes APR-tilting for
hereditary algebras.

Cluster tilting theory furnishes a way to construct abelian categories from triangulated categories.
Koenig and Zhu \cite[Theorem 3.2]{KZ} showed that if $\C$ is a triangulated category and $\X$ is a cluster tilting  subcategory of $\C$, then
the quotient category $\C/\X$ is an abelian category. Moreover,  the category $\C/\X$ is Gorenstein of Gorenstein dimension at most one \cite[Theorem 4.3]{KZ}, which implies that it is either hereditary or of infinite global dimension.

In \cite{AIR}, Adachi, Iyama and Reiten introduced a generalized classical tilting theory, which is called $\tau$-tilting theory.  They  proved that for a $2$-Calabi-Yau triangulated category $\C$ with a cluster tilting object $T$, there exists a
 bijection between the basic cluster tilting objects in $\C$ and the basic support $\tau$-tilting modules in $\mod\End_\C(T)^{\textrm{op}}$ (which is equivalent to the quotient category $\C/\add T[1]$). This bijection was generalized first by Yang and Zhu \cite{YZ} by introducing the notion of relative cluster tilting objects in a triangulated category with a cluster tilting object, later by Fu, Geng and Liu \cite{FGL}  by introducing the notion of relative rigid objects in a triangulated category $\C$ with a rigid object.
Later, Iyama, J{\o}rgensen and Yang \cite{IJY} gave a functor version of $\tau$-tilting theory. They consider modules over a category and showed for a triangulated category $\C$ with a silting subcategory $\mathcal S$, there exists a bijection between
the set of two-term silting subcategories of $\C$ and the set of support $\tau$-tilting subcategories of $\mod\mathcal S$.
For a triangulated category $\C$ with a cluster tilting subcategory $\T$,  Yang, Zhou and Zhu \cite{YZZ} introduced the notion of (weak) $\T[1]$-cluster tilting subcategories of $\C$. They showed that there exists a bijection between the set of weak $\T[1]$-cluster tilting subcategories of $\C$ and the set of support $\tau$-tilting subcategories of $\mod\T$ which generalizes the bijection in \cite{YZ}. This bijection is an analogue to that in \cite{IJY}. Recently, Zhou and Zhu \cite{ZZ3} introduced the notion of two-term weak $\R[1]$-cluster tilting subcategories for a triangulated category $\C$ with a rigid subcategory $\R$, and showed that two-term weak $\R[1]$-cluster tilting subcategories correspond bijectively with support $\tau$-tilting subcategories. This result unifies the bijections given in \cite{IJY} and in \cite{YZZ}.

When we combine all these results, we can get a hint that support $\tau$-tilting subcategories in the quotient category may have some relations with relative rigid subcategories in the original category. Hence it is reasonable to investigate what subcategory is a support $\tau$-tilting subcategory in the quotient category related to, under a more general setting.
We want our results to be valid not only on triangulated categories, but also on exact categories.
In this article we  work on extriangulated categories introduced by Nakaoka and Palu in \cite{NP}. It is a simultaneous generalization of
exact categories and triangulated categories,  while there are some other examples of extriangulated categories which are neither exact nor triangulated (see \cite{NP,ZZ1}).
Our aim is to build the relation between relative cluster tilting theory in extriangulated categories and $\tau$-tilting theory.

In this article, let $k$ be a field and $\B$ be a Hom-finite, Krull-Schmidt, $k$-linear extriangulated category with enough projectives $\mathcal P$ and enough injectives $\mathcal I$ (please see Section 2 for more details of extriangulated category). Let $\R$ be a contravariantly finite rigid subcategory of $\B$ which contains all projective objects.

Our main results on relative cluster tilting theory and $\tau$-tilting theory can be
summed up as follows.

\begin{thm} {\rm (see Theorem \ref{main1} and Theorem \ref{main2})}
Let $\X$ be a subcategory of $\h=\CoCone(\R,\R)$ {\rm (}see Definition \ref{def1} for details{\rm )}.
Under some suitable conditions, we show that
\begin{itemize}
\item[\rm (1)] $\X$ is relative rigid if and only if $\overline \X$ is
a $\tau$-rigid subcategory of $\h/\R$;
\smallskip

\item[\rm (2)] $\X$ is maximal relative rigid if and only if $\overline \X$ is
a support $\tau$-tilting subcategory of $\h/\R$.
\smallskip

%\item[\rm (3)] any support tilting subcategory in $\h/\R$ is support $\tau$-tilting.
\end{itemize}

%In addition, when $\R$ is cluster tilting and $\B/\R$ is hereditary, we discuss the relation between tilting subcategories in $\B/\R$ and cluster titling subcategories in $\B$.
\end{thm}

In \cite{IT}, Ingalls and Thomas introduced the concept of support tilting modules and established a bijection between cluster tilting and classical tilting. Later, support tilting subcategories were introduced by Holm and J{\o}rgensen, which is a generalization of support tilting modules.  In \cite[Theorem 3.5]{HJ}, they give a bijection between support tilting subcategories and weak cluster tilting subcategories (which is a special kind of maximal relative rigid subcategories) under certain assumptions. In this article, we show that support tilting subcategories are support $\tau$-tilting, and their bijection becomes a special case of our results.

\begin{thm}{\rm (see Theorem \ref{main2.5})}
In the abelian quotient category $\h/\R$,
\begin{itemize}
\item[(1)] any support tilting subcategory is support $\tau$-tilting;
\smallskip

\item[(2)] any functorially finite support $\tau$-tilting subcategory $\overline \X$ with $\Ext^2_{\h/\R}(\overline \X,-)=0$ is support tilting.
\end{itemize}
\end{thm}

When $\R$ is cluster tilting, we discuss the relation between tilting subcategories in $\B/\R$ and cluster titling subcategories in $\B$.

\begin{thm}{\rm (see Theorem \ref{main4})}
Let $\B$ be a Frobenius extriangulated category and $\R$ be a cluster tilting subcategory. Assume that $\B/\R$ is hereditary. Then for any subcategory $\M$ such that $\M\cap\R=\mathcal P$, if $\M$ is a cluster tilting subcategory, then $\M$ is a tilting subcategory in $\B/\R$.
\end{thm}

We have some interesting applications of our results, one is the following:

\begin{prop}{\rm (see Proposition \ref{completion3})}
Let $\B$ be a triangulated category and $\R$ be cluster tilting subcategory. Let $\U$ be a contravariantly finite rigid subcategory such that $\U\cap\R={0}$. If projective dimension of any object of $\U$ is less than 2 in $\B/\R$, then $\U$ is contained in a tilting subcategory in $\B/\R$.
\end{prop}

This article is organized as follows. In Section 2, we introduce some necessary background of extriangulated category. In Section 3, we give a one-to-one correspondence between $\tau$-rigid subcategories of $\h/\R$ and relative rigid subcategories of $\h$. In Section 4, we first show that any $\tau$-rigid subcategory of $\h/\R$ which satisfy certain condition is contained in a support $\tau$-tilting subcategory, then we give a one-to-one correspondence between support $\tau$-tilting subcategories of $\h/\R$ and maximal relative rigid subcategories of $\h$. We also show that any support tilting subcategory in $\h/\R$ is a support $\tau$-tilting subcategory. In Section 5, we study the relation between tilting subcategories in $\B/\R$ and cluster titling subcategories $\B$ when $\R$ is cluster tilting. In Section 6, we give some applications of our main results, which mainly concentrate on the case when $\B$ is a triangulated category. In Section 7, we give two examples to explain our results.

\section{Preliminaries}

\subsection{Extriangulated categories}
Let us briefly recall the definition and some basic properties of extriangulated category from \cite{NP}.
We omit some details here, but the readers can find
them in \cite{NP}.

Let $\B$ be an additive category equipped with an additive bifunctor
$$\mathbb{E}: \B^{\rm op}\times \B\rightarrow {\rm Ab},$$
where ${\rm Ab}$ is the category of abelian groups. For any objects $A, C\in\B$, an element $\delta\in \mathbb{E}(C,A)$ is called an $\mathbb{E}$-extension.
Let $\mathfrak{s}$ be a correspondence which associates with an equivalence class $$\mathfrak{s}(\delta)=\xymatrix@C=0.8cm{[A\ar[r]^x
 &B\ar[r]^y&C]}$$ to any $\mathbb{E}$-extension $\delta\in\mathbb{E}(C, A)$. This $\mathfrak{s}$ is called a {\it realization} of $\mathbb{E}$, if it makes the diagrams in \cite[Definition 2.9]{NP} commutative.
 A triplet $(\B, \mathbb{E}, \mathfrak{s})$ is called an {\it extriangulated category} if it satisfies the following conditions.
\begin{enumerate}
\setlength{\itemsep}{2.5pt}
\item $\mathbb{E}\colon\B^{\rm op}\times \B\rightarrow \rm{Ab}$ is an additive bifunctor.

\item $\mathfrak{s}$ is an additive realization of $\mathbb{E}$.

\item $\mathbb{E}$ and $\mathfrak{s}$  satisfy the compatibility conditions in \cite[Definition 2.12]{NP}.
 \end{enumerate}
\smallskip

We collect some basic concepts, which can be found in \cite{NP}.

\begin{defn}\label{dein}
Let $(\B,\EE,\mathfrak{s})$ be an extriangulated category.
\begin{itemize}
\setlength{\itemsep}{2.5pt}
\item[{\rm (1)}] A sequence $A\xrightarrow{~x~}B\xrightarrow{~y~}C$ is called a {\it conflation} if it realizes some $\EE$-extension $\del\in\EE(C,A)$. In this case, $x$ is called an {\it inflation} and $y$ is called a {\it deflation}.

\item[{\rm (2)}] If a conflation $A\xrightarrow{~x~}B\xrightarrow{~y~}C$ realizes $\delta\in\mathbb{E}(C,A)$, we call the pair $( A\xrightarrow{~x~}B\xrightarrow{~y~}C,\delta)$ an {\it $\EE$-triangle}, and write it in the following way.
$$A\overset{x}{\longrightarrow}B\overset{y}{\longrightarrow}C\overset{\delta}{\dashrightarrow}$$
We usually do not write this $``\delta"$ if it is not used in the argument.
\item[{\rm (3)}] Let $A\overset{x}{\longrightarrow}B\overset{y}{\longrightarrow}C\overset{\delta}{\dashrightarrow}$ and $A^{\prime}\overset{x^{\prime}}{\longrightarrow}B^{\prime}\overset{y^{\prime}}{\longrightarrow}C^{\prime}\overset{\delta^{\prime}}{\dashrightarrow}$ be any pair of $\EE$-triangles. If a triplet $(a,b,c)$ realizes $(a,c)\colon\delta\to\delta^{\prime}$, then we write it as
$$\xymatrix{
A \ar[r]^x \ar[d]^a & B\ar[r]^y \ar[d]^{b} & C\ar@{-->}[r]^{\del}\ar[d]^c&\\
A'\ar[r]^{x'} & B' \ar[r]^{y'} & C'\ar@{-->}[r]^{\del'} &}$$
and call $(a,b,c)$ a {\it morphism of $\EE$-triangles}.

\item[{\rm (4)}] An object $P\in\B$ is called {\it projective} if
for any $\EE$-triangle $A\overset{x}{\longrightarrow}B\overset{y}{\longrightarrow}C\overset{\delta}{\dashrightarrow}$ and any morphism $c\in \Hom_\B(P,C)$, there exists $b\in \Hom_\B(P,B)$ satisfying $yb=c$.
We denote the subcategory of projective objects by $\mathcal P$. Dually, we can define injective object and the subcategory of injective objects is denoted by $\I$.

\item[{\rm (5)}] We say that $\B$ {\it has enough projective objects} if
for any object $C\in\B$, there exists an $\EE$-triangle
$A\overset{x}{\longrightarrow}P\overset{y}{\longrightarrow}C\overset{\delta}{\dashrightarrow}$
satisfying $P\in\mathcal P$. Dually we can define $\B$ {\it having enough injective objects}.

%\item[{\rm (6)}] Let $\mathcal{X}$ be a subcategory of $\B$. We say $\mathcal{X}$ is {\it extension closed}
%if a conflation $A\rightarrow B\rightarrow C$ satisfies $A,C\in\mathcal{X}$, then $B\in\mathcal{X}$.
\end{itemize}
\end{defn}

In this article, let $k$ be a field and $(\B,\mathbb{E},\mathfrak{s})$ be a Krull-Schmidt, Hom-finite, $k$-linear extriangulated category with enough projectives $\mathcal P$ and enough injectives $\mathcal I$. When we say $\D$ is a subcategory of $\B$, we always mean that $\D$ is closed under isomorphisms, direct sums and direct summands.
\smallskip

By \cite{NP}, we give the following useful remark, which will be used later in the proofs.

\begin{rem}\label{useful}
Let $\xymatrix@C=0.5cm@R0.5cm{A\ar[r]^a &B \ar[r]^b &C \ar@{-->}[r] &}$ and $\xymatrix@C=0.5cm@R0.5cm{X\ar[r]^x &Y \ar[r]^y &Z \ar@{-->}[r] &}$ be two $\EE$-triangles. Then
\begin{itemize}
\item[(a)] In the following commutative diagram
$$\xymatrix{
X\ar[r]^x \ar[d]_f &Y \ar[d]^g \ar[r]^y &Z \ar[d]^h \ar@{-->}[r] &\\
A\ar[r]^a &B \ar[r]^b &C \ar@{-->}[r] &}
$$
$f$ factors through $x$ if and only if $h$ factors through $b$.
\item[(b)] In the following commutative diagram
$$\xymatrix{
A\ar[r]^a \ar[d]_s &B \ar[d]^r \ar[r]^b &C \ar[d]^t \ar@{-->}[r] &\\
X\ar[r]^x \ar[d]_f &Y \ar[d]^g \ar[r]^y &Z \ar[d]^h \ar@{-->}[r] &\\
A\ar[r]^a &B \ar[r]^b &C \ar@{-->}[r] &}
$$
$fs=1_A$ implies $B$ is a direct summand of $C\oplus Y$ and $C$ is a direct summand of $Z\oplus B$; $ht=1_C$ implies $A$ is a direct summand of $A\oplus Y$ and $A$ is a direct summand of $X\oplus B$.
\item[(c)] Let $f\colon A\rightarrow D$ be any morphism. Then we have a morphism of $\EE$-triangles
$$\xymatrix{
A \ar[r]^{x} \ar[d]_f &B \ar[r]^{y} \ar[d]^g &C \ar@{=}[d]\ar@{-->}[r]&\\
D \ar[r]_{d} &E \ar[r]_{e} &C\ar@{-->}[r] &
}
$$
Moreover, the sequence $A\xrightarrow{\binom{f}{x}}D\oplus B\xrightarrow{(d~-g)}E\dashrightarrow$ is also an $\EE$-triangle.
\end{itemize}
\end{rem}

\begin{defn}\label{def1}
Let $\B'$ and $\B''$ be two subcategories of $\B$.
\begin{itemize}
\item[(a)] Denote by $\CoCone(\B',\B'')$ the subcategory
$$\{X\in \B \text{ }|\text{ } \textrm{there exists an}~ \text{ } \EE\text{-triangle } \xymatrix@C=0.5cm@R0.5cm{ X \ar[r] &B' \ar[r] &B'' \ar@{-->}[r] &} \text{, }B'\in \B' \text{ and }B''\in \B'' \}.$$
\item[(b)] Denote by $\Cone(\B',\B'')$ the subcategory
$$\{X\in \B \text{ }|\text{ } \textrm{there exists an}~ \text{ } \EE\text{-triangle } \xymatrix@C=0.5cm@R0.5cm{B' \ar[r] &B'' \ar[r] &X \ar@{-->}[r] &} \text{, }B'\in \B' \text{ and }B''\in \B''  \}.$$
\item[(c)] Let $\Omega \B'=\CoCone(\mathcal P,\B')$. We write an object $D$ in the form $\Omega B$ if it admits an $\EE$-triangle $\xymatrix@C=0.5cm@R0.5cm{D \ar[r] &P \ar[r] &B \ar@{-->}[r] &}$ where $P\in \mathcal P$.
\item[(d)] Let $\Sigma \B'=\Cone(\B',\mathcal I)$. We write an object $D'$ in the form $\Sigma B'$ if it admits an $\EE$-triangle $\xymatrix@C=0.5cm@R0.5cm{B' \ar[r] &I \ar[r] &D' \ar@{-->}[r] & }$ where $I\in \mathcal I$.
\end{itemize}

\end{defn}

\begin{defn}\cite[Definition 2.10]{ZZ2}
 Let $\C$ be a subcategory of $\B$.
\begin{itemize}
\item[\rm (1)] $\C$ is called rigid if $\EE(\C,\C)=0$;

\item[\rm (2)] $\C$ is called \emph{cluster tilting}, if it satisfies the
following conditions:
\begin{itemize}

\item[\rm (a)] $X\in\C$ if and only if $\EE (X, \C) = 0$;

\item[\rm (b)] $X\in\C$ if and only if $\EE(\C, X) = 0$;

\item[\rm (c)] $\C$ is functorially finite.
\end{itemize}

\end{itemize}
\end{defn}

From this definition, we know that if $\C$ is cluster tilting, then it contains all the projectives and all the injectives. Moreover, $\CoCone(\C,\C)=\Cone(\C,\C)=\B$.

%\begin{defn}
%A rigid subcategory $\R$ of $\B$ is called \emph{fully rigid} if it admits a cotorsion pair $(\R,\K)$ and $\R\neq\mathcal P$.
%From the definition, we know that any fully rigid subcategory is contravariantly finite.
%\end{defn}
%\smallskip
%
%As the following example shows, such subcategories appear naturally.

%\begin{exm} Let $\B$ be an extriangulated category with enough projectives $\mathcal P$ and enough injectives $\mathcal I$.
%\begin{itemize}
%\item[(1)] Any cluster tilting subcategory is fully rigid.
%\vspace{1mm}
%
%\item[(2)] If $\R$ is a contravariantly finite rigid subcategory of $\B$ such that $\mathcal P\subsetneqq\R$,
%then $(\R, \R^{\bot_1})$ is a cotorsion pair. Thus $\R$ is fully rigid.
%\end{itemize}
%\end{exm}

%\begin{rem}  If $\R$ is a contravariantly finite rigid subcategory of $\B$ such that $\mathcal P\subseteq\R$,

%\end{rem}

\smallskip

In this article, we always assume $\R$ is a \textbf{contravariantly finite rigid} subcategory of $\B$ such that $\mathcal P\subsetneq \R$, unless otherwise specified.

\begin{lem}\label{summand}
$\CoCone(\R,\R)$ and $\Omega \R$ are closed under direct summands.
\end{lem}

\begin{proof}
We only show the first one, the second is similar.\\
Let $X_1\oplus X_2\in \CoCone(\R,\R)$. It admits an $\EE$-triangle $\xymatrix{X_1\oplus X_2 \ar[r]^-{\svech{q_1}{q_2}} &R^1 \ar[r]^{p} &R^2 \ar@{-->}[r] &}$ where $R^1,R^2\in \R$. Since $q_1=(q_1~q_2)\binom{1}{0}$, we have that $q_1$ is an inflation. So $q_1$ admits an $\EE$-triangle $\xymatrix{X_1 \ar[r]^{q_1} &R^1 \ar[r]^{p_1} &R \ar@{-->}[r] &}$. Thus we get the following commutative diagram:
$$\xymatrix{
X_1 \ar[d]_-{\svecv{1}{0}} \ar[r]^{q_1} &R^1 \ar@{=}[d] \ar[r]^{p_1} &R \ar[d] \ar@{-->}[r] & \\
X_1\oplus X_2 \ar[r]^-{\svech{q_1}{q_2}} \ar[d]_-{\svech{1}{0}} &R^1 \ar[r]^p \ar[d]^a &R^2 \ar[d] \ar@{-->}[r] &\\
X_1 \ar[r]^{q_1} &R^1 \ar[r]^{p_1} &R \ar@{-->}[r] &.
}
$$
Therefore $R$ is a direct summand of $R^1\oplus R^2\in \R$. Then $R\in \R$ and we have $X_1\in \CoCone(\R,\R)$. Hence $\CoCone(\R,\R)$ is closed under direct summands.
\end{proof}

For objects $A,B\in\B$ and a subcategory $\B_0$ of $\B$, let $[\B_0](A,B)$ be the subgroup of $\Hom_{\B}(A,B)$ consisting of morphisms which factor through objects in $\B_0$. %For a morphism $x:A\to X$ (or $x:X\to B$), let $[x](A,B)$ be the subgroup of $\Hom_{\B}(A,B)$ consisting of morphisms which factor through $x$, let $[\B',x](A,B)=[\B'](A,B)\cap [x](A,B)$.  For another morphism $x':A\to X$ (or $x':X\to B$), let $[x,x'](A,B)=[x](A,B)\cap [x'](A,B)$.

Let $\B'$ be a subcategory of $\B$.
\begin{itemize}
\item We denote by $\underline \B'$ the category which has the same objects as $\B'$, and $$\Hom_{\uB'}(A,B)=\Hom_{\B}(A,B)/[\mathcal P](A,B)$$ where $A,B\in \B'$. For any morphism $f\in \Hom_{\B}(A,B)$, we denote its image in $\Hom_{\uB'}(A,B)$ by $\underline f$. $\uB'$ is a subcategory of $\uB$.
\item We denote by $\overline \B'$ the category which has the same objects as $\B'$, and $$\Hom_{\oB'}(A,B)=\Hom_{\B}(A,B)/[\R](A,B)$$ where $A,B\in \B'$. For any morphism $f\in \Hom_{\B}(A,B)$, we denote its image in $\Hom_{\oB'}(A,B)$ by $\overline f$. $\oB'$ is a subcategory of $\oB$.
\end{itemize}
%We denote $\B'/\mathcal P$ by $\underline \B'$ if $\mathcal P\subseteq \B'\subseteq \B$. For any morphism $f\colon A\to B$ in $\B$, we denote by $\underline{f}$ the image of $f$ under the natural quotient functor $\B\to \uB$. We denote the image of a subcategory $\B'$ under the natural quotient functor $\pi:\B\to \B/\R$ by $\overline {\B'}$. For any morphism $g\colon A\to B$ in $\B$, we denote by $\overline{g}$ the image of $g$ under $\pi$.

%Let $\W$ and $\s$ be subcategories of $\B$. We denote
%$$\W_\s=\{W\in \W \text{ }|\text{ }W \text{ has no non-zero direct summands in }\s \}.$$
Let $\h=\CoCone(\R,\R)$.
\begin{prop}\label{b1}
We have the following properties:
\begin{itemize}
\item[(a)] $\overline \h\simeq \mod\underline {\Omega \R}\simeq \mod\underline \R$, they are abelian categories.
\item[(b)] The subcategory $\underline {\Omega \R}$ is the enough projectives in $\overline \h$.
\item[(c)] There is a additive functor $H:\B\to \overline \h$ such that
\begin{itemize}
\item[(i)] any $\EE$-triangle $\xymatrix{A\ar[r] &B\ar[r] &C\ar@{-->}[r] &}$ admits an exact sequence $H(A) \to H(B)\to H(C)$ in $\overline \h$;
\item[(ii)] $H(R)=0$ if $R\in \R$;
\item[(iii)]  $H(B)=B$ if $B\in \h$, $H(f)=\overline f$ if $f$ is a morphism in $\h$.
\end{itemize}
%\item[(c)] $\B_\K=\h_\C$.
\end{itemize}
\end{prop}

\begin{proof}
(a) Since $\R$ is a contravariantly finite rigid subcategory of $\B$, then $(\R, \R^{\bot_1})$ is a cotorsion pair in $\B$ (see \cite[Definition 4.1]{NP} and \cite[Lemma 2.6]{LZ0}). According to \cite[Theorem 3.2]{LN}, $\overline \h$ is abelian.

The equivalence $\overline \h\simeq \mod\underline {\Omega \R}$ is proved in \cite[Theorem 1.2]{LZ}.

The equivalence $\overline \h\simeq \mod\underline {\R}$ is proved in \cite[Theorem 3.4]{ZZ2}.

Hence we have $\overline \h\simeq \mod\underline {\Omega \R}\simeq \mod\underline \R$.
\smallskip

(b) Note that a morphism $f$ in $\Omega \R$ factors through $\R$ if and only if it factors through $\mathcal P$. By \cite[Theorem 4.10]{LN}, we get that the subcategory $\underline {\Omega \R}$ of $\overline \h$ is the enough projectives.

(c) This is followed by \cite[Theorem 3.5, Corollary 3.8]{LN}.
%\smallskip

%(c) Since $\h\cap \K=\C$, by the definition of a fully rigid subcategory, we have $\B_\K=\h_\C$.
\end{proof}

The following lemma is very useful.

\begin{lem}\label{useful}
Let $A\to B\to R\dashrightarrow$ be an $\EE$-triangle where $A\in \h$ and $R\in \R$. Then $B\in \h$.
\end{lem}

\begin{proof}
Since $A\in \h$, it admits an $\EE$-triangle $A\to R_1\to R_2\dashrightarrow$ where $R_1,R_2\in \R$. Then we have the following commutative diagram
$$\xymatrix{
A \ar[r] \ar[d] &B \ar[r] \ar[d] &R\ar@{=}[d] \ar@{-->}[r]&\\
R_1 \ar[r] \ar[d] &R_1\oplus R \ar[r] \ar[d] &R \ar@{-->}[r]&\\
R_2 \ar@{-->}[d] \ar@{=}[r] &R_2 \ar@{-->}[d]\\
&&
}
$$
Hence by definition $B\in \h$.
\end{proof}

\section{Relative rigid subcategories and $\tau$-rigid subcategories}

\begin{defn}\label{d1}
For two objects $M,N\in\B$, denote by $\overline {[\R]}(M,\Sigma N)$ the subset of $[\R](M,\Sigma N)$ such that $\alpha\in \overline {[\R]}(M,\Sigma N)$ if we have the following commutative diagram:
$$\xymatrix@C=0.8cm@R0.6cm{
&M\ar[r]^h \ar@{.>}[d] &R_0 \ar[d]^j\\
N \ar[r] &I_N \ar[r]^i &\Sigma N \ar@{-->}[r] &
}
$$
where $jh=\alpha$, $I_N\in \mathcal I$ and $R_0\in \R$.

A subcategory $\X$ of $\B$ is called relative rigid (with respect to $\R$) if $\overline {[\R]}(\X,\Sigma\X)={[\R]}(\X,\Sigma \X)$. An object $X$ is called relative rigid if $\add X$ is relative rigid.
\end{defn}

For convenience, a relative rigid subcategory in this article is also called $\R$-rigid.

\begin{lem}
Any rigid subcategory is $\R$-rigid.
\end{lem}

\proof Let $\X$ be a rigid subcategory. For any $X,X'\in\X$ and $\alpha \in [\R](X,X')$, we have the following diagram:
$$\xymatrix@C=0.8cm@R0.6cm{
&X\ar[r]^h  &R_0 \ar[d]^j\\
X'\ar[r]^p &I_{X'} \ar[r]^{i\;\;} &\Sigma X' \ar@{-->}[r] &
}
$$
where $jh=\alpha$, $R_0\in \R$ and $I_{X'}\in \mathcal I$.
Applying the functor $\Hom_{\B}(X,-)$ to the $\EE$-triangle
$$\xymatrix{X'\ar[r]^p&I_{X'} \ar[r]^{i\;\;} &\Sigma X' \ar@{-->}[r] &},$$
we get the following exact sequence:
$$\Hom_{\B}(X,I_{X'})\xrightarrow{\Hom_{\B}(X,~i)}\Hom_{\B}(X,\Sigma X')\xrightarrow{~~}\EE(X,X')=0.$$
Then there exists a morphism $g\colon X\to I_{X'}$ such that $ig=jh$.
This shows that $\X$ is $R$-rigid.  \qed

\begin{prop}\label{sum}
Let $\X$ be $\R$-rigid and $\R_0\subseteq \R$. Then $\add(\X\cup \R_0)$ is $\R$-rigid if and only if $\EE(\R_0,\X)=0$.
\end{prop}

\begin{proof}
Let $X_0\in \X$ and $R_0\in \R$ and $y\in \Hom_{\B}(R_0, \Sigma X_0)$ be any morphism. We have the following diagram
$$\xymatrix{
&&R_0 \ar[d]^y\\
X_0 \ar[r]^{j_0} &I \ar[r]^{i_0} &\Sigma X_0  \ar@{-->}[r] &
}
$$
If $\add(\X\cup \R_0)$ is $\R$-rigid, by the definition of $\R$-rigid, $y$ factors through $i_0$, which implies that $\EE(\R_0,\X)=0$.\\
If $\EE(\R_0,\X)=0$, then from the diagram above, we have $[\R](\R_0,\Sigma \X)=[\overline \R](\R_0,\Sigma \X)$. It is enough to check that $[\R](\X, \Sigma \R_0)=[\overline \R](\X, \Sigma \R_0).$ But this is followed by the fact that $\R$ is rigid.
\end{proof}

We give the following definition, which is a special case of \cite[Definition 1.3 (ii)]{IJY}.

\begin{defn}
A subcategory $\overline \X\subseteq \overline \h$ is called $\tau$-rigid if for any object $X\in \overline \X$, there is an exact sequence $\Omega R_0'\xrightarrow{\overline{f'}} \Omega R_1'\to X\to 0$ where $\Hom_{\overline \h}(\overline {f'}, X')$ is surjective for any $X'\in \overline \X$.
\end{defn}

Let $\C$ and $\D$ be subcategories of $\B$. We denote
$$\C_\D=\{C\in \C \text{ }|\text{ }C \text{ has no non-zero direct summands in }\D \}.$$

For any object $X\in\h_\R$, there exists an $\EE$-triangle
$$X\xrightarrow{~h~}R_0\xrightarrow{~u~}R_1\dashrightarrow,$$
where $R_0,R_1\in\R$. Since $\B$ has enough projectives, there exists an $\EE$-triangle
$$\Omega R_0\xrightarrow{~h~}P\xrightarrow{~u~}R_0\dashrightarrow,$$
where $P\in\mathcal P$. By (ET4)$^{\rm op}$, we have the following commutative diagram $(\maltese)$
$$\xymatrix{
\Omega R_0 \ar@{=}[r] \ar[d]^f  &\Omega R_0 \ar[d]^{p_1}\\
\Omega R_1 \ar[r]^b \ar[d]^g &P \ar[r] \ar[d]^{q_1} &R_1 \ar@{=}[d] \ar@{-->}[r] &\\
X \ar[r]^h \ar@{-->}[d] &R_0 \ar[r] \ar@{-->}[d] &R_1 \ar@{-->}[r] &\\
& & &
}
$$
of $\EE$-triangles. Then we get the following exact sequence in $\overline \h$
$$ \Omega R_0\xrightarrow{~\overline f~}\Omega R_1\xrightarrow{~\overline g~}X\rightarrow 0. $$

\begin{lem}\label{suj}
Let $X\in \h_\R$. In the diagram $(\maltese)$, $\Hom_{\B}(f, R)$ is surjective for any $R\in \R$. Moreover, if $\Hom_{\overline \h}(\overline f, X')$ is surjective for an object $X'\in \h$, then $\Hom_{\B}(f, X')$ is also surjective.
\end{lem}

\begin{proof}
Let $r:\Omega R_0\to R$ be any morphism. Since $\EE(R_0,R)=0$, $r$ factors through $p_1$, hence it also factors through $f$, which implies that $\Hom_{\B}(f, R)$ is surjective. If $\Hom_{\overline \h}(\overline f, X')$ is surjective where $X'\in \h$, let $x_0:\Omega R_0\to X'$ be any morphism. Then there is a morphism $x_1:\Omega R_1\to X'$ such that $\overline x_0=\overline {x_1f}$. Hence $x_0-x_1f$ factors through $\R$, and by the preceding argument, we get that $x_0-x_1f$ factors through $f$, hence $\Hom_{\B}(f, X')$ is also surjective.
\end{proof}

\begin{lem}
Let $X\in \h_\R$. In the diagram $(\maltese)$, $g$ is a right $(\Omega \R)$-approximation.
\end{lem}

\begin{proof}
Let $\Omega R'\in \Omega \R$ and $r':\Omega R'\to X$ be any morphism. Since $hr'$ factors through $\mathcal P$, it also factors through $q_1$. Hence there is a morphism $r'':\Omega R'\to \Omega R_1$ such that $r'=gr''$, which implies that $g$ is right $(\Omega \R)$-approximation.
\end{proof}

\begin{lem}\label{1}
Let $\overline \X\subseteq \overline \h$ be a $\tau$-rigid subcategory. Then for any object $X\in \X_\R$, in the diagram $(\maltese)$, we have that $\Hom_{\B}(f, X_0)$ is surjective for an object $X_0\in \X$.
\end{lem}

\begin{proof}
Since $X\in \X_\R$ where $\overline \X$ is $\tau$-rigid, $X$ admits an exact sequence $\Omega R_0'\xrightarrow{\overline f'} \Omega R_1'\to X\to 0$ where $\Hom_{\overline \h}(\overline{f'}, X')$ is surjective for any $X'\in \overline \X$. $f'$ admits the following commutative diagram
$$\xymatrix{
\Omega R_0' \ar[r]^{p_0} \ar[d]_{f'} &P' \ar[r] \ar[d] & R_0' \ar@{=}[d] \ar@{-->}[r] &\\
\Omega R_1' \ar[r] &X' \ar[r] &R_0' \ar@{-->}[r] &
}
$$
where $P'\in \mathcal P$ and $X'\in \h$. We have $X'\simeq X$ in $\overline \h$. Since $X$ has no a direct summand in $\R$, we have that $X$ is a direct summand of $X'$. We have the following commutative diagram
$$\xymatrix{
\Omega R_0' \ar@{=}[r] \ar[d]_-{\svecv{f'}{p_0}}  &\Omega R_0'\ar[d]^-{\svecv{p_2f'}{p_0}} \\
\Omega R_1'\oplus P' \ar[d]_{\alpha} \ar[r]^-{\left(\begin{smallmatrix}
p_2&0\\
0&1
\end{smallmatrix}\right)} &P_2\oplus P' \ar[r]_-{\svech{q_2}{0}} \ar[d] &R_1' \ar@{=}[d] \ar@{-->}[r]&\\
X' \ar[r]_{h'} \ar@{-->}[d] &R_0'' \ar[r] \ar@{-->}[d] &R_1' \ar@{-->}[r] &\\
& &
}
$$
where $P_2\in \mathcal P$, $R_0'',R_1'\in \R$. By the similar argument as in the preceding lemma, we get that $\alpha$ is a right $(\Omega \R)$-approximation. Now we have the following commutative diagram
$$\xymatrix{
\Omega R_0 \ar[r]^{f} \ar[d]^{y} &\Omega R_1 \ar[r]^g \ar[d]^{\eta} &X \ar[d]^{x} \ar@{-->}[r] &\\
\Omega R_0' \ar[r]^-{\svecv{f'}{p_0}=\gamma} \ar[d]^{y'} &\Omega R_1'\oplus P' \ar[r]^-{\alpha} \ar[d] &X' \ar[d]^{x'} \ar@{-->}[r] &\\
\Omega R_0 \ar[r]^{f}  &\Omega R_1 \ar[r]^g  &X  \ar@{-->}[r] &\\
}
$$
where $x'x=1_X$. Then there is a morphism $z:\Omega R_1\to \Omega R_0$ such that $y'y+zf=1$. Let $X_0\in \X$ and $a:\Omega R_0\to X_0$ be any morphism. Then we have a morphism $ay':\Omega R_0'\to X_0$. Since $\Hom_{\overline \h}(\overline{f'}, X_0)$ is surjective, by the similar argument as in Lemma \ref{suj}, there is a morphism $\beta:\Omega R_1'\oplus P'\to X_0$ such that $ay'=\beta \gamma$. Thus $a=a(y'y+zf)=\beta \gamma y+azf=(\beta \eta +az)f$, which implies that $\Hom_{\B}(f,X_0)$ is surjective.
\end{proof}

We show the main theorem of this section.

\begin{thm}\label{main1}
Let $\X$ be a subcategory of $\h$.  Then $\X$ is $\R$-rigid if and only if $\overline \X$ is
a $\tau$-rigid subcategory of $\overline \h$.
 \end{thm}

\begin{proof}
Let $X\in \X$. Then we have the following commutative diagram
$$\xymatrix{
\Omega R_0 \ar@{=}[r] \ar[d]^f  &\Omega R_0 \ar[d]^{p_1}\\
\Omega R_1 \ar[r]^b \ar[d]^g &P \ar[r] \ar[d] &R_1 \ar@{=}[d] \ar@{-->}[r] &\\
X \ar[r]^h \ar@{-->}[d] &R_0 \ar[r] \ar@{-->}[d] &R_1 \ar@{-->}[r] &\\
& & &
}
$$
of $\EE$-triangles which induces an exact sequence in $\overline \h$
$$\Omega R_0\xrightarrow{\overline f}\Omega R_1\xrightarrow{\overline g}X\rightarrow 0. $$
%Let $x':\Omega R_0\to X'$ be any morphism where $X'\in \X$.
Let $X'\in \X$. Since $\B$ has enough injectives, $X'$ admits an $\EE$-triangle
$$\xymatrix@C=0.8cm@R0.6cm{X'\ar[r] &I_{X'} \ar[r]^{i\;\;} &\Sigma X' \ar@{-->}[r] &}$$ where $I_{X'}\in \mathcal I$. Let $j:R_0\to \Sigma X'$ (resp. $x':\Omega R_0\to X'$) be any morphism. Then we get the following commutative diagram
$$\xymatrix{
\Omega R_0 \ar@{=}[r] \ar[d]^f  &\Omega R_0 \ar[d]^{p_1} \ar[r]^{x'} &X' \ar[d]\\
\Omega R_1 \ar[r]^b \ar[d]^g &P \ar[r] \ar[d] &I_{X'} \ar[d]^i \\
X \ar[r]^h \ar@{-->}[d] &R_0 \ar[r]^j \ar@{-->}[d] &\Sigma X' \ar@{-->}[d]\\
& & & &
}
$$
of $\EE$-triangles.

When $\X$ is $\R$-rigid, we have $jh$ factors through $i$ and then $x'$ factors through $f$.
Thus $\Hom_{\overline \h}(\overline f, X')$ is surjective. This shows that $\overline{\X}$ is $\tau$-rigid.

%we get the following projective presentation:
%$$\Hom_{\oB}(\Omega\R,\Omega R_0) \xrightarrow{\Hom_{\oB}(\Omega\R,\underline f)} \Hom_{\oB}(\Omega\R,\Omega R_1) \xrightarrow{\Hom_{\oB}(\Omega\R,\underline g)} \Hom_{\oB}(\Omega\R,X)\to 0$$
%in $\mod\underline{\Omega\R}$.
%By Lemma \ref{b1}(ii), for any $X'\in\X$, we obtain the following commutative diagram
%$$\xymatrix{
%\Hom_{\oB}(\Omega R_1,X') \ar[d]_-{\Hom_{\oB}(\underline f, X')} \ar[rr]^-{\simeq} &&\Hom_{\mod\underline{\Omega\R}}(\Hom_{\oB}(\Omega\R,\Omega R_1),\Hom_{\oB}(\Omega\R,X'))\ar[d]^-{\Hom_{\mod\underline{\Omega\R}}(\Hom_{\oB}(\Omega\R, \underline f),\Hom_{\oB}(\Omega\R,X'))}\\
%\Hom_{\oB}(\Omega R_0,X) \ar[rr]^-{\simeq} &&\Hom_{\mod\underline{\Omega\R}}(\Hom_{\oB}(\Omega\R, \Omega R_0),\Hom_{\oB}(\Omega\R,X'))
%}
%$$
%Then $\Hom_{\oB}(\Omega\R,\X)$ is
%a $\tau$-rigid subcategory of $\mod\underline{\Omega\R}$ if and only if $\Hom_{\oB}(\underline f, X')$
%s an epimorphism.
%\vspace{-2mm}

Now assume that $\overline \X$ is
$\tau$-rigid. By Lemma \ref{1} $\Hom_{B}(f, X')$ is epimorphism, then there exists a morphism $x_1\colon\Omega R_1\to X'$ such that $x'=x_1f$. Hence $jh$ factors through $i$. Hence by definition $\X$ is $\R$-rigid.
\end{proof}

%Conversely, we assume that $\X$ is $\R$-rigid, we will show that $\Hom_{\B}(f,X')$ is epimorphism.  If it is like this, we obtain that $\Hom_{\oB}(\underline f,X')$ is epimorphism. Thus $\Hom_{\oB}(\Omega\R,\X)$ is
%$\tau$-rigid.
%Let $x:\Omega R_1\to X$ be any morphism. Since $\X$ is $\R$-rigid, we have that $jh$ factors through $i$, which implies $x$ factors through $f$ by \cite[Corollary 3.5]{NP}. This shows that $\Hom_{\B}(f,X')$ is surjective.\qed
At last, we give the following lemma, which is a special case of \cite[Lemma 5.2]{IJY}.%(note that since $\mod (\Omega \C/\mathcal P)$ is abelian, when apply the proof of \cite[Lemma 5.2]{IJY}, many objects in $\Mod (\Omega \C/\mathcal P)$ actually lie in $\mod (\Omega \C/\mathcal P)$).

\begin{lem}\label{fac}
If $\overline \X\subseteq \overline \h$ is $\tau$-rigid, then $\Ext^1_{\overline \h}(\overline \X,\Fac \overline \X )=0$, where $\Fac \overline \X$ is the subcategory of quotients of objects in $\overline \X$.
\end{lem}

\section{Support $\tau$-tilting subcategories and maximal relative rigid subcategories}

We give the following definition, which is a special case of \cite[Definition 1.3 (iv)]{IJY}.

\begin{defn}
A pair $(\overline \W,\overline \s)$ is called a support $\tau$-tilting pair of $\overline \h$ if the following conditions are satisfied:
\begin{itemize}
\item[(a)] $\W\subseteq \h_\R$ is an $\R$-rigid subcategory, $\s=\{\Omega R\in \Omega \R \text{ }|\text{ }\Hom_{\uB}(\Omega R,\W)=0 \}$.
\item[(b)] For any object $U\in \overline {\Omega \R}$, there exists an exact sequence $U\xrightarrow{\overline f} W_1\to W_2\to 0$ where $W_1,W_2\in \overline \W$, and $\overline f$ is a left $\overline \W$-approximation of $B$.
\end{itemize}
We call $\overline \W$ a support $\tau$-tilting subcategory if it admits a support $\tau$-tilting pair $(\overline \W,\overline \s)$.
\end{defn}

\begin{rem}
Since $\Hom_{\uB}(\Omega \R,-)\simeq \EE(\R,-)$, we have $\s=\{\Omega R\in \Omega \R \text{ }|\text{ }\EE( R,\W)=0 \}$.
\end{rem}

\begin{prop}\label{supportcluster}
Let $\X\subseteq \h$ be an $\R$-rigid subcategory such that $\X\cap \R=\{R\in \R \text{ }|\text{ }\EE(R,\X)=0 \}$. If any object $\Omega R\in \Omega \R$ admits an $\EE$-triangle $\xymatrix{\Omega R\ar[r]^f &X^1 \ar[r] &X^2\ar@{-->}[r] &}$ where $X^1,X^2\in \X$ and $f$ is a left $\X$-approximation, then $\overline \X$ is support $\tau$-tilting. Especially, when $\R$ is cluster tilting and $\X$ is a cluster tilting subcategory such that $\X\neq \R$, we have $\overline \X$ is a support $\tau$-tilting subcategory in $\overline \h$ (=$\oB$).
\end{prop}

\begin{proof}
Since $\X$ is $\R$-rigid, $\overline \X$ is $\tau$-rigid.
By assumption, any object $\Omega R\in \Omega \R$ admits a commutative diagram of $\EE$-triangles
$$\xymatrix{
\Omega R \ar[r]^f \ar@{=}[d] &X^1 \ar[d] \ar[r] &X^2 \ar[d] \ar@{-->}[r] &\\
\Omega R \ar[r] &P \ar[r] &R \ar@{-->}[r] &
}
$$
where $X^1,X^2\in \X$, $P\in \mathcal P$, $R\in \R$ and $f$ is a left $\X$-approximation. Then it admits an exact sequence $\Omega R\xrightarrow{\overline f} X^1\to X^2\to 0$ in $\overline \h$, where $\overline f$ is a left $\overline \X$-approximation. Hence by definition $\overline \X$ is support $\tau$-tilting.

When $\R$ is cluster tilting, we have $\h=\B$. When $\X$ is a cluster tilting subcategory such that $\X\neq \R$, we get that it is $\R$-rigid since cluster tilting subcategories are rigid. By the definition of cluster tilting subcategory, we have $\X\cap \R=\{R\in \R \text{ }|\text{ }\EE(R,\X)=0 \}$. Hence $\overline \X$ is a support $\tau$-tilting subcategory in $ \oB$.
\end{proof}

\begin{lem}\label{rep}
Let $\overline \W$ be a support $\tau$-tilting subcategory in $\overline \h$ such that $\W\subseteq \h_\R$. Let $\T=\{R\in \R \text{ }|\text{ }\EE(R,\W)=0 \}$ and $\X=\add(\W\cup \T)$. Then any object $R\in \R$ admits an $\EE$-triangle $$\xymatrix{X_1\ar[r] &X_2\ar[r]^{r} &R\ar@{-->}[r] &}$$
where $X_1,X_2\in \X$. Moreover, $r$ is a right $\X$-approximation of $R$.
\end{lem}

\begin{proof}
For any object $R\in \R$, $\Omega R$ admits an exact sequence $\Omega R \xrightarrow{\overline t_1} W_1 \to W_2\to 0$ where $W_1,W_2\in \W$ and $\overline t_1$ is a left $\overline \W$-approximation of $\Omega R$. We have the following commutative diagram
$$\xymatrix{
\Omega R \ar[r]^q \ar[d]^{t_1} &P\ar[r] \ar[d]^{p} &R \ar@{=}[d] \ar@{-->}[r] &\\
W_1 \ar[r]^w &X_2 \ar[r]^r &R \ar@{-->}[r] &.
}
$$
By Lemma \ref{useful} $X_2\in \h$. Let $X_2=W_2'\oplus R'$ where $W_2'\in \h_\R$ and $R'\in \R$. Then $W_2'\simeq W_2$. We need to show $R'\in \T$. It is enough to prove that for any morphism $c: R'\to \Sigma X$ where $X\in \X$, in the following commutative diagram
$$\xymatrix{
\Omega R' \ar[r]^{p'} \ar[d]^a &P' \ar[r] \ar[d] &R' \ar[d]^c \ar@{-->}[r] &\\
X\ar[r] &I \ar[r]^i &\Sigma X \ar@{-->}[r] &
}
$$
$c$ factors through $i$. Hence it is enough to show $\overline {[\R]}(X_2, X)={[\R]}(X_2,X)$.\\
Let $f\in {[\R]}(X_2,\Sigma X)$. There is a morphism $j_1:W_1\to I$ such that $fw=ij_1$. Since $P$ is projective, there is a morphism $p':P\to I$ such that $ip'=fp$. Hence we have the following commutative diagram
$$\xymatrix{
\Omega R\ar[r]^-{\svecv{q}{t_1}} \ar@{.>}[d]_t &P\oplus W_1 \ar[r]^-{\svech{-p}{w}} \ar[d]^-{\svech{-p'}{j_1}} &X_2 \ar[d]^f \ar@{-->}[r] &\\
X \ar[r] &I \ar[r]_i &\Sigma X \ar@{-->}[r] &
}
$$
Since $\overline t_1$ is a left $\overline \W$-approximation, we can get that $\Omega R \xrightarrow{\svecv{q}{t_1}} P\oplus W_1$ is a left $\X$-approximation. Hence $t$ factors through $\svecv{q}{t_1}$, which implies $f$ factors through $i$.\\
Let $x:X\to R$ be any morphism, then in the following commutative diagram
$$\xymatrix{
&&X \ar[d]^x\\
X_1\ar[r] \ar@{=}[d] &X_2\ar[r]^{r} \ar[d] &R\ar@{-->}[r] \ar[d]^{r'} &\\
X_1 \ar[r] &I_1 \ar[r]^{i_1} &\Sigma X_1 \ar@{-->}[r] &
}
$$
$r'x$ factors through $i_1$. Hence $x$ factors through $r$, which means $r$ is a right $\X$-approximation of $R$.
%It is not hard to check that $r$ is a right $\X$-approximation, we leave the proof to the readers.
\end{proof}

By this Lemma, we can get the following corollary.

\begin{cor}\label{repcor}
Let $\overline \W$ be a support $\tau$-tilting subcategory in $\overline \h$ such that $\W\subseteq \h_\R$. Let $\T=\{R\in \R \text{ }|\text{ }\EE(R,\W)=0 \}$ and $\X=\add(\W\cup \T)$. If $R\in \R$ admits an $\EE$-triangle $\xymatrix{X_1\ar[r] &X_2\ar[r]^{r} &R\ar@{-->}[r] &}$ where $r$ is a right $\X$-approximation, then $X_1\in \X$.
\end{cor}

\begin{defn}
Let $\X\subseteq \h$ be $\R$-rigid. $\X$ is called maximal $\R$-rigid if for any subcategory $\X'\subseteq \h$, $\add(\X\cup \X')$ is $\R$-rigid implies $\X'\subseteq \X$. An object $X\in\h$ is called maximal $\R$-rigid if $\add(\{X \}\cup\mathcal P)$ is a maximal $\R$-rigid subcategory.
\end{defn}

\begin{prop}\label{completion}
Let $\X$ be a subcategory of $\h$ such that $\X\cap \R=\{R\in \R \text{ }|\text{ }\EE(R,\X)=0 \}$. If $\X$ is an $\R$-rigid subcategory of $\h$ such that every object in $\Omega \R$ admits a left $\X$-approximation, then $\overline \X$ is contained in a support $\tau$-tilting subcategory $\{ Y\in \Fac \overline \X \text{ }|\text{ } \Ext^1_{\oB}(Y, \Fac \overline \X)=0 \}=:\mathbf{P}(\Fac \overline \X)$.
\end{prop}

\begin{proof}
Let $\X$ be an $\R$-rigid subcategory of $\h$ such that every object in $\Omega \R$ admits a left $\X$-approximation. Then every object $\Omega R\in \Omega \R$ admits a commutative diagram
$$\xymatrix{
\Omega R \ar[r] \ar[d]^{f'} &P \ar[r]^p \ar[d] &R \ar@{-->}[r] \ar@{=}[d] &\\
X' \ar[r] &U \ar[r] &R \ar@{-->}[r] &
}
$$
where $P\in \mathcal P$, $R\in \R$ and $f'$ is a left $\X$-approximation. By Lemma \ref{useful} $U\in \h$. This diagram induces an $\EE$-triangle
$$\xymatrix{\Omega R \ar[r]^-{\alpha} &X'\oplus P \ar[r] &U \ar@{-->}[r] &}$$
where $\alpha$ is still a left $\X$-approximation. Let $\U$ be the subcategory of objects $U$ which admits
%We show that $U\in \X$, since $\overline \X$ is $\tau$-rigid by Theorem \ref{main1}, by Proposition \ref{supportcluster} $\overline \X$ is support $\tau$-tilting. \\
%It is enough to show that $\add(\X\cup \{U\})$ is $\R$-rigid. For convenience, we can consider the following
an $\EE$-triangle
$$\xymatrix{\Omega R \ar[r]^{f} &X \ar[r]^g &U \ar@{-->}[r] &}$$
where $f$ is a left $\X$-approximation. Let $\Y=\{ X\oplus U \text{ }|\text{ } X\in \X,U\in \U \}$. We show that $\add \overline \Y=\mathbf{P}(\Fac \overline \X).$\\
We first show that $\add \overline \Y\subseteq \mathbf{P}(\Fac \overline \X)$. Since $\mathbf{P}(\Fac \overline \X)$ is closed under direct summands, we only need to show $\overline \Y\subseteq \mathbf{P}(\Fac \overline \X)$. By Lemma \ref{fac}, we have $\overline \X\subseteq \mathbf{P}(\Fac \overline \X)$. Since $\mathbf{P}(\Fac \overline \X)$ is closed under direct sums, it is enough to show that $\U\subseteq \mathbf{P}(\Fac \overline \X)$.\\
By the definition, any object $U\in \U$ admits an $\EE$-triangle $\xymatrix{\Omega R \ar[r]^{f} &X \ar[r]^g &U \ar@{-->}[r] &}$ where $f$ is a left $\X$-approximation. It induces an exact sequence $\Omega R \xrightarrow{\overline f} X \xrightarrow{\overline g} U\to 0$ in $\overline \h$. Since $\Omega R$ is a projective object in $\overline \h$, $\overline f$ is in fact a left $(\Fac \overline \X)$-approximation. Let $\Omega R \xrightarrow{\overline f_1} V\xrightarrow{\overline f_2} X$ be an epic-monic factorization of $\overline f$. Then in the short exact sequence $0\to V\xrightarrow{\overline f_2} X\xrightarrow{\overline g} U\to 0$, $\overline f_2$ is a left $(\Fac \overline \X)$-approximation. Hence we have
$$\Hom_{\oB}(X,\Fac \overline \X)\to \Hom_{\oB}(V,\Fac \overline \X)\xrightarrow{0} \Ext^1_{\oB}(U,\Fac \overline \X)\to \Ext^1_{\oB}(X,\Fac \overline \X)=0$$
which implies that $\Ext^1_{\oB}(U,\Fac \overline \X)=0$. Since $U\in \Fac \overline \X$, we have $U\in \mathbf{P}(\Fac \overline \X)$.\\
Now we show that $\add \overline \Y\supseteq \mathbf{P}(\Fac \overline \X)$. It is enough to consider an indecomposable object $Y\in \mathbf{P}(\Fac \overline \X)$ such that $Y\notin \X$.\\
Since $Y\in \Fac \overline \X$, it admits an epimorphism $\overline y:X_0\to Y\to 0$ where $X_0\in \overline \X$. Then by \cite[Corollary 2.26]{LN}, we have a commutative diagram
$$\xymatrix{
X_0 \ar[r]^{r_1} \ar[d]^y &R^1 \ar[d] \ar[r] &R^2 \ar@{=}[d] \ar@{-->}[r] &\\
Y \ar[r] &R_0 \ar[r] &R^2 \ar@{-->}[r] &
}
$$
where $R_0,R^1,R^2\in \R$. Then we get a commutative diagram
$$\xymatrix{
\Omega R_0 \ar[r]^{p_0} \ar[d]^{x_0} &P_0 \ar[r] \ar[d] &R_0 \ar[d]  \ar@{-->}[r] &\\
X_0 \ar[r]_-{\svecv{y}{r_1}} &Y\oplus R^1 \ar[r] &R_0 \ar@{-->}[r] &
}
$$
which induces an $\EE$-triangle $\xymatrix{\Omega R_0 \ar[r]^-{\svecv{x_0}{p_0}=\alpha} &X_0\oplus P_0 \ar[r] &Y\oplus R^1\ar@{-->}[r] &}$. To show $Y\oplus R^1\in \U$, we need to prove that $\alpha$ is a left $\X$-approximation. This $\EE$-triangle induces an exact sequence $\Omega R_0 \xrightarrow{\overline x_0} X_0\xrightarrow{\overline y} Y\to 0$ in $\overline \h$, it is enough to show that $\overline x_0$ is a left $\overline \X$-approximation.\\
Let $\Omega R_0 \xrightarrow{\overline x_1} W\xrightarrow{\overline x_2} X$ be an epic-monic factorization of $\overline x_0$. Then for the short exact sequence $0\to W\xrightarrow{\overline x_2} X_0\xrightarrow{\overline y} Y\to 0$, we have
$$\Hom_{\oB}(X_0,\Fac \overline \X)\xrightarrow{\Hom_{\oB}(\overline x_2,\Fac \overline \X)} \Hom_{\oB}(W,\Fac \overline \X)\to \Ext^1_{\oB}(Y,\Fac \overline \X)=0$$
which implies that $\overline x_2$ is a left $(\Fac \overline X)$-approximation. Hence $\overline x_0$ is a left $\overline \X$-approximation.
%By definition $\Y$ is $\R$-rigid, then we get $\overline \Y$ is support $\tau$-tiling.

%Let $\Y=\{ X\oplus U \text{ }|\text{ } X\in \X,U\in \U \}$.

We show that $\mathbf{P}(\Fac \overline \X)$ is support $\tau$-tilting. It is enough to show that $\add \Y$ is $\R$-rigid, since then by Proposition \ref{supportcluster} and Theorem \ref{main1}, $\mathbf{P}(\Fac \overline \X)$ is support $\tau$-tilting.\\
We only need to show $[\R](\Y,\Sigma Y)=[\overline \R](\Y,\Sigma Y)$. We first show that $\EE(R_0,U)\simeq \Hom_{\uB}(\Omega R_0,U)=0$ for any $R_0\in \T$ and $U\in \U$. Let $a:\Omega R_0\to U$ be any morphism. Since we have the following commutative diagram
$$\xymatrix{
&&\Omega R_0 \ar[d]^a\\
\Omega R \ar[r]^{f} \ar@{=}[d] &X \ar[r]^g \ar[d] &U \ar[d]^b \ar@{-->}[r] &\\
\Omega R \ar[r]  &P \ar[r]^p &R \ar@{-->}[r]&
}$$
where $ba$ factors through $\mathcal P$, we get that $ba$ also factors through $p$. Then there is a morphism $c:\Omega R_0\to X$ such that $a=gc$. But $0=\EE(R_0,X)\simeq \Hom_{\uB}(\Omega R_0,X)$, hence $\underline a=0$ and we can get  $\Hom_{\uB}(\Omega R_0,U)=0$.\\
Let $X_0\in \X$. For any morphism $u\in [\R](U,\Sigma X_0)$, since $ug\in [\R](X,\Sigma X_0)=[\overline \R](X,\Sigma X_0)$, we have the following commutative diagram
$$\xymatrix{
\Omega R \ar[r]^{f} \ar[d]^x &X \ar[r]^g \ar[d] &U \ar[d]^u \ar@{-->}[r] &\\
X_0 \ar[r] &I_0 \ar[r]^{j_0} &\Sigma X_0 \ar@{-->}[r] &.
}
$$
Since $f$ is a left $\X$-approximation, $x$ factors through $f$, then $u$ factors through $j_0$. Hence $u\in [\overline \R](U,\Sigma X_0)$ and we can get $[\R](U,\Sigma X_0)=[\overline \R](U,\Sigma X_0)$. This also shows that $(\add \Y)\cap \R\subseteq \X\cap\R$.\\
For any morphism $v\in [\R](X_0,\Sigma U)$, we have the following commutative diagram
$$\xymatrix{
\Omega X_0 \ar[r]^{q_0} \ar[d]^{c_0} &P_0 \ar[r]^{p_0} \ar[d] &X_0 \ar[d]^{a_0} \ar@{-->}[r] &\\
\Omega R' \ar[r] \ar[d]^{d_0} &P' \ar[r] \ar[d] &R' \ar[d]^{b_0} \ar@{-->}[r] &\\
U \ar[r] &I_U \ar[r]^{j_U} &\Sigma U \ar@{-->}[r] &
}
$$
where $b_0a_0=v$, $R'\in \R$, $P_0,P'\in\mathcal P$ and $I_U\in \mathcal I$. By the similar argument, we can show that there is a morphism $e:\Omega R'\to X$ such that $d_0=ge$. Then we have the following commutative diagram
$$\xymatrix{
\Omega X_0 \ar[r]^{q_0} \ar[d]^{c_0} &P_0 \ar[r]^{p_0} \ar[d] &X_0 \ar[d]^{a_0} \ar@{-->}[r] &\\
\Omega R' \ar[r] \ar[d]^{e} &P' \ar[r] \ar[d] &R' \ar[d]^{b_0'} \ar@{-->}[r] &\\
X \ar[r] &I \ar[r]^{j} &\Sigma X \ar@{-->}[r] &.
}
$$
Since $b_0'a_0$ factors through $j$, $ec_0$ factors through $q_0$. Hence $d_0c_0=gec_0$ factors through $q_0$, which implies that $b_0a_0$ factors through $j_U$. Thus we get $[\R](X_0,\Sigma U)=[\overline \R](X_0,\Sigma U)$.\\
For any morphism $w\in [\R](U'',\Sigma U)$ where $U''$ also lies in $\U$, since $[\R](X'',\Sigma U)=[\overline \R](X'',\Sigma U)$ where $X''\in \X$, we have the following commutative diagram
$$\xymatrix{
\Omega R'' \ar[r]^{f} \ar[d]^r &X'' \ar[r]^g \ar[d] &U'' \ar[d]^w \ar@{-->}[r] &\\
U \ar[r] &I_{U} \ar[r]^{j_U} &\Sigma U \ar@{-->}[r] &.
}
$$
Then there is a morphism $s:\Omega \R''\to X$ such that $r=gs$. But $f$ is a left $\X$-approximation, hence $s$ factors through $f$, then $r$ factors through $f$, which implies that $w$ factors through $j_U$. Thus $[\R](U'',\Sigma U)=[\overline \R](U'',\Sigma U)$.
\end{proof}

%\begin{prop}
%Let $\Y$ be the subcategory defined in Lemma \ref{completion}. We have $$\add \overline \Y=\{ Y\in \Fac \overline \X \text{ }|\text{ } \Ext^1_{\oB}(Y, \Fac \overline \X)=0 \}=:\mathbf{P}(\Fac \overline \X).$$
%\end{prop}

%\begin{proof}
%Let $\X$, $\U$ be the same subcategories as in Lemma \ref{completion}.
%Now let $0\neq Z\in \mathbf{P}(\Fac \overline \X)$ be an indecomposable object. Then it admits an epimorphism $\overline x:X\to Z$ where $X\in \overline \X$. Moreover, we have the following commutative diagram:
%$$\xymatrix{
%\Omega R \ar[r] \ar[d] &P \ar[r] \ar[d] &R \ar@{=}[d] \ar@{-->}[r] &\\
%X \ar[r]^x &Z \ar[r] &R \ar@{-->}[r] &
%}
%$$
%where $R\in \R$ and $P\in \mathcal P$.
%\end{proof}

Now we can show the main theorem of this section.

\begin{thm}\label{main2}
Let $\X$ be a subcategory of $\h$ such that $\X\cap \R=\{R\in \R \text{ }|\text{ }\EE(R,\X)=0 \}$.  Then $\overline \X$ is a support $\tau$-tilting subcategory of $\overline \h$ if and only if $\X$ is a maximal $\R$-rigid subcategory of $\h$ such that every object in $\Omega \R$ admits a left $\X$-approximation.
\end{thm}

\begin{proof}
Let $\X=\add(\W\cup \T)$ such that $\W\subseteq \h_\R$ and $\T=\X\cap \R$. We show the ``only if" part first.\\
Let $\overline \W$ be a support $\tau$-tilting subcategory in $\overline \h$. By the proof of Lemma \ref{rep}, every object in $\Omega \R$ admits a left $\X$-approximation. We show that $\X$ is maximal $\R$-rigid, it is enough to check that for any object $Y\in \h_\R$, if $\overline {[\R]}(Y,\Sigma \X)={[\R]}(Y,\Sigma \X)$ and $\overline {[\R]}( \X,\Sigma Y)={[\R]}( \X,\Sigma Y)$, then $Y\in \X$.\\
$Y$ admits an $\EE$-triangle $\xymatrix{Y \ar[r]^{r_1} &R_1 \ar[r]^{r_2} &R_2 \ar@{-->}[r] &}$ where $R_1,R_2\in \R$. By Lemma \ref{rep}, $R_1$ admits an $\EE$-triangle $\xymatrix{X_1\ar[r] &X_0\ar[r]^{r} &R_1\ar@{-->}[r] &}$ where $X_1,X_0\in \X$ and $r$ is a right $\X$-approximation. Then we have the following commutative diagram
$$\xymatrix{
X_1 \ar@{=}[r] \ar[d] &X_1 \ar[d]\\
U \ar[r] \ar[d] &X_0 \ar[r]^{r'} \ar[d]^r &R_2 \ar@{=}[d] \ar@{-->}[r] &\\
Y \ar[r]^{r_1} \ar@{-->}[d] &R_1 \ar[r]^{r_2} \ar@{-->}[d] &R_2 \ar@{-->}[r] &.\\
& &
}
$$
Let $f:X\to R_2$ be any morphism where $X\in \X$, we have the following commutative diagram
$$\xymatrix{
&X \ar@{.>}[dr]^-j \ar@{.>}[ddr] \ar[drrr]^-f\\
Y \ar[rr] \ar@{=}[d] &&R_1 \ar[d] \ar[rr]^-{r_2} &&R_2 \ar[d] \ar@{-->}[r] &\\
Y \ar[rr] &&I \ar[rr] && \Sigma Y \ar@{-->}[r] &
}
$$
where $I\in \mathcal I$. Since $r$ is a right $\X$-approximation, $j$ factors through $r$, hence $f$ factors through $r'$. This implies $r'$ is also a right $\X$-approximation. By Corollary \ref{repcor}, we have $U\in \X$. In the following commutative diagram
$$\xymatrix{
X_1 \ar@{=}[r] \ar[d] &X_1 \ar[d] \ar@{=}[r] \ar[d] &X_1 \ar[d]\\
U \ar[r] \ar[d] &X_0 \ar[r] \ar[d]^r &I' \ar[d]^{i'} \\
Y \ar[r]^{r_1} \ar@{-->}[d] &R_1 \ar[r]^{s} \ar@{-->}[d] &\Sigma X_1 \ar@{-->}[d]\\
& & &
}
$$
$sr_1$ factors through $i'$, hence the first column splits, which means $Y$ is a direct summand of $U\in \X$. Hence $Y\in \X$.\\
Now we show the ``if" part. Let $\X$ be a maximal $\R$-rigid subcategory of $\h$ such that every object in $\Omega \R$ admits a left $\X$-approximation. By Proposition \ref{completion}, $\overline \X$ is contained in a support $\tau$-tiling subcategory $\overline \Y$ such that $\Y\cap\R=\{R\in \R \text{ }|\text{ }\EE(R,\Y)=0 \}$. Then $\Y$ is a maximal $\R$-rigid subcategory which contains $\X$, we have $\X=\Y$. Hence $\overline \X$ is support $\tau$-tiling.
%We get the following corollary immediately.

%\begin{cor}
%For any $\R$-rigid subcategory $\U\subseteq \h_\R$, there exists a maximal $\R$-rigid subcategory in $\h$ which contains
%$\U$.
%\end{cor}

%\begin{proof}
%We have $\overline \U\subseteq \mathrm{P}(\Fac \overline \U)$ where $\mathrm{P}(\Fac \overline \U)$ is a support $\tau$-tilting subcategory. By Theorem \ref{main2}, we can get a  maximal $\R$-rigid subcategory in $\h$ which contains $\U$.
\end{proof}

We can get the following corollary immediately.

\begin{cor}\label{completion2}
If $\X$ is an $\R$-rigid subcategory of $\h$ such that every object in $\Omega \R$ admits a left $\X$-approximation, then $\X$ is contained in a maximal $\R$-rigid subcategory.
\end{cor}

Compare with \cite[Theorem 4.4]{ZZ3}, we have the following interesting observation.

According to the proof of Theorem \ref{main2} and Lemma \ref{rep}, we have the following proposition.

\begin{prop}\label{maximal}
Let $\R$ be a contravariantly finite rigid subcategory such that $\mathcal P\subsetneq \R$. Then $\X\subseteq \h$ is maximal $\R$-rigid if and only if the following conditions are satisfied:
\begin{itemize}
\item[(a)] $\R\subseteq \Cone(\X,\X)$.
\item[(b)] $\X=\{M\in \h \text{ }|\text{ } [\R](M,\Sigma \X)=[\overline \R](M,\Sigma \X) \text{ and } [\R](\X,\Sigma M)=[\overline \R](\X,\Sigma M)  \}$
\end{itemize}
\end{prop}

\begin{proof}
According to the proof of Theorem \ref{main2} and Lemma \ref{rep}, any maximal $\R$-rigid subcategory $\X\subseteq \h$ satisfies (a) and (b).\\
We show that if $\X\subseteq \h$ satisfies conditions (a) and (b), then $\X$ is maximal $\R$-rigid.\\
From (b) we know that $\X$ is $\R$-rigid such that $\X\cap\R=\{R\in \R \text{ }|\text{ }\EE(R,\X)=0 \}$. Since $\R\subseteq \Cone(\X,\X)$, any object $R\in \R$ admits a commutative diagram
$$\xymatrix{
\Omega R \ar[r] \ar[d]^f &P \ar[r] \ar[d] &R \ar@{=}[d] \ar@{-->}[r] &\\
X_1 \ar[r] &X_0 \ar[r] &R \ar@{-->}[r] &
}
$$
which induces an $\EE$-triangle $\xymatrix{\Omega R \ar[r]^{\alpha\quad} &X_1\oplus P \ar[r] &X_0 \ar@{-->}[r] &}$ where $\overline \alpha=\overline f$. Let $r:\Omega R\to X$ be any morphism where $X\in \X$. Then we have the following commutative diagram
$$\xymatrix{
\Omega R \ar[r]^{\alpha\quad} \ar@{=}[d] &X_1\oplus P \ar[d] \ar[r] &X_0 \ar@{-->}[r] \ar[d]^a &\\
\Omega R \ar[r] \ar[d]^r &P \ar[r] \ar[d] &R \ar[d]^b \ar@{-->}[r] &\\
X \ar[r] &I \ar[r]^i &\Sigma X \ar@{-->}[r] &
}$$
where $I\in \mathcal I$. Since $\X$ is $\R$-rigid, we have $ba$ factors through $i$. Hence $r$ factors through $\alpha$. This means any object $\Omega R$ admits an exact sequence $\Omega R \xrightarrow{\overline f} X_1 \to X_0\to 0$ where $\overline f$ is a left $\overline \X$-approximation. By the definition, $\overline \X$ is support $\tau$-tilting. Thus by Theorem \ref{main2}, $\X$ is maximal $\R$-rigid.
\end{proof}

\subsection{Support $\tau$-tilting subcategories vs support tilting subcategories}

Support tilting subcategories were introduced by Holm and J{\o}rgensen, which can be regarded as a generalization of support tilting modules.
\begin{defn}\cite[Definition 2.1]{HJ}
To say that $\M$ is a support tilting subcategory of an abelian category $\A$ means that $\M$ is a full subcategory which
\begin{itemize}
\setlength{\itemsep}{2.5pt}
\item is closed under direct sums and direct summands;

\item is functorially finite in $\A$;

\item satisfies $\Ext^2_{\A}(\M,-)=0$;

\item satisfies $\Ext^1_{\A}(\M,\M)=0$;

\item satisfies that if $A$ is a subquotient of an object from $\M$ such that
$\Ext^1_{\A}(\M,A)=0$, then $B$ is a quotient of an object from $\M$.
\end{itemize}
%An object $M$ is called a support tilting object if $\add M$ is support tilting.
\end{defn}

We have the following theorem.

\begin{thm}\label{main2.5}
In the abelian quotient category $\overline \h$,
\begin{itemize}
\item[(a)] any support tilting subcategory $\overline \X$ is support $\tau$-titling;
\item[(b)] any functorially finite support $\tau$-tilting subcategory $\overline \X$ which satisfies $\Ext^2_{\overline \h}(\overline \X,-)=0$ is support tilting.
\end{itemize}
\end{thm}

\begin{proof}
In this article, the subcategories we mentioned are always assumed to be closed under direct summands and sums.

(a) Let $\overline \X\subseteq \overline \h$ be support tilting. By definition we can get that $\overline \X$ is $\tau$-rigid. Hence by Lemma \ref{completion}, there is a support $\tau$-tilting subcategory $\overline \Y\subseteq \Fac \overline \X$ that contains $\overline \X$. Then any object $Y\in \overline \Y$ admits a short exact sequence $0\to V\to X_0\xrightarrow{\overline y} Y\to 0$ where $X_0\in \overline \X$. Since $\overline \X$ is functorially finite, we can assume that $\overline y$ is a right $\overline \X$-approximation. Hence $\Ext^1_{\oB}(\overline \X,V)=0$. Since $V$ is a subquotient of $X_0\in \overline \X$, we have $V\in \Fac \overline \X$. Since $\Ext^1_{\oB}(\overline \X,\Fac \overline \X)=0$ by Lemma \ref{fac}, the preceding short exact sequence splits, hence $Y\in\overline \X$. We have $\overline \X=\overline \Y$, which means $\overline \X$ support $\tau$-titling.

(b) Now let $\overline \X\subseteq \overline \h$ be a functorially finite support $\tau$-tilting subcategory which satisfies $\Ext^2_{\overline \h}(\overline \X,-)=0$. Any object $X\in \overline \X$ admits a short exact sequence $0\to \Omega R^1\xrightarrow{\overline r} \Omega R^2\to X\to 0$ where $\Hom_{\oB}(\overline r,X')$ is surjective for any $X'\in \overline \X$. Hence we have $\Ext^1_{\oB}(X,X')=0$, which implies that $\Ext^1_{\oB}(\overline \X,\overline \X)=0$. Let $0\neq Y\in \overline \h$ such that $\Ext^1_{\oB}(\overline \X, Y)=0$ and $Y$ is a subquotient of an object $ X\in \overline {\X}$. Then we have an epimorphism and a monomorphism $\xymatrix{X \ar@{->>}[r]^{\overline v} &T & Y \ar@{ >->}[l]_{\overline y}}$. Since $\overline {\X}$ is functorially finite, we can assume that $\overline v$ is a right $\overline {\X}$-approximation of $T$. Then we have the following commutative diagram of short exact sequences in $\overline \h$
$$\xymatrix{
&&0 \ar[d] &0 \ar[d]\\
0 \ar[r] & W \ar[r] \ar@{=}[d] & U \ar[r]^{\overline u} \ar[d]^{\overline u'} & Y \ar[r] \ar[d]^{\overline y} &0\\
0 \ar[r] & W \ar[r] & X \ar[r]^{\overline v} \ar[d]^{\overline z_1} & T \ar[r] \ar[d]^{\overline z_2} &0\\
&& Z \ar@{=}[r] \ar[d] & Z \ar[d]\\
&&0 &0
}
$$
where $\Ext^1_{\oB}(\overline \X, W)=0$. Since $\Ext^1_{\oB}(\overline \X, Y)=0$, we have $\Ext^1_{\oB}(\overline \X, U)=0$. Moreover, we have that $\overline z_1$ is a right $\overline \X$-approximation.
%We show that $\overline z_1$ is a right $\overline {\V}$-approximation.\\
%Let $v':V'\to Z$ be any morphism where $V'\in \V$. Since $\Ext^1_{\oB}(\overline \V, Y)=0$, there is a morphism $\overline t:  V'\to  T$ such that $\overline v'= \overline {z_2t}$. Since $\overline v$ is a right $\overline {\V}$-approximation of $T$, there is a morphism $\overline v''\colon V'\to V$ such that $\overline t=\overline {vv''}$. Hence $\overline v'=\overline {z_2vv''}=\overline {z_1v''}$.\\
Since $\Omega R$ is projective, we have the following commutative diagram:
$$\xymatrix{
&\Omega R \ar[r]^{\overline r} \ar[d]^{\overline a} \ar[ddl]_{\overline b} &X^1\ar[r]^{\overline s} &X^2 \ar[r] &0\\
 &Y \ar[d]^{\overline y}\\
X \ar[r]^{\overline v} &T\ar[r] &0
}
$$
where $\overline r$ is a left $\X$-approximation and $\overline a$ is an epimorphism. Then there is a morphism $\overline c:X^1 \to X$ such that $\overline b=\overline {cr}$. We have $\overline {z_2vcr}=\overline {z_2ya}=0$, hence there is a morphism $\overline s':X^2\to Z$ such that $\overline {z_2vc}=\overline {s's}$. Since $\overline z_1$ is a right $\overline \X$-approximation, there is a morphism $\overline c':X^2\to X$ such that $\overline s'=\overline {z_1c'}$. Hence $0=\overline {z_2vc}-\overline {z_1c's}=\overline z_1(\overline {c}-\overline {c's})=0$. Thus there is a morphism $\overline s'':X^1\to U$ such that $\overline {c}-\overline {c's}=\overline {u's''}$. Then $\overline {yus''r}=\overline{vu's''r}=\overline {vcr}-\overline {vc'sr}=\overline {vb}=\overline {ya}$. Hence $\overline a=\overline {us''r}$, and $\overline {us''}:X^1\to Y$ is an epimorphism, which means $Y\in \Fac \overline \X$.
%Since $\overline z_1$ is a right $\overline \X$-approximation, there is a morphism $c':X^1\to X$ such that $\overline {tvc}$
%Since $\Ext^1_{\oB}(V^2,V)$=0, There is a morphism $\overline c:V^1 \to V$ such that $\overline {u'b}=\overline {cr}$. Then $\overline {z_2vc}=$
%there is a morphism $\overline c:X^1 \to X$ such that $\overline b=\overline {cr}$. Then we get $\overline a=\overline {ucr}$ where $\overline {uc}:V^1\to Y$ is an epimorphism.
\end{proof}

\begin{rem}
Theorem \ref{main2.5}(b) actually holds on any abelian category with enough projectives.
%\begin{itemize}
%\item[(1)] Theorem \ref{main2.5}(b) actually holds on any abelian category with enough projectives.
%\item[(2)] This theorem combining with Theorem \ref{main2} is a generalization of \cite[Theorem 3.5]{HJ}.
%\end{itemize}
\end{rem}

%\begin{cor}
%If $\B$ is a triangulated categories with shift functor $[1]$, then
%\end{cor}

\section{Support $\tau$-tilting subcategories and cluster tilting subcategories}

In this section, we assume that $\R$ is a cluster tilting subcategory of $\B$. An subcategory $\oB'$ of $\oB$ is called hereditary if for any object $A\in \oB'$, $\Ext^2_{\oB}(A,-)=0$.
\vspace{1mm}

\begin{defn}
A subcategory $\overline \T\subseteq \oB$ is called a tilting subcategory if it satisfies the following conditions:
\begin{itemize}
\item[(a)] $\overline \T$ is hereditary.
\item[(b)] $\Ext^1_{\oB}(\overline\T,\overline\T)=0$.
\item[(c)] Any object $\Omega R\in \overline {\Omega \R}$ admits a short exact sequence $0\to \Omega R\to T_1\to T_2\to 0$ where $T_1,T_2\in \overline \T$.
\end{itemize}
\end{defn}

$\B$ is called Frobenius if $\mathcal P=\mathcal I$. We will show the following theorem.

\begin{thm}\label{main4}
Let $\B$ be a Frobenius extriangulated category and $\R$ be a cluster tilting subcategory of $\B$.
If $\oB$ is hereditary, then the image $\overline\M$ of any cluster tilting subcategory $\M$ such that  $\M\cap \R=\mathcal P$ is a tilting subcategory. %Moreover, $\overline \M$ is also a cotilting subcategory.
%\begin{itemize}
%\item[(a)]
%\item[(b)] any contravariantly finite subcategory $\M\subseteq \B_\R$ which is a tilting subcategory in $\oB$, is a cluster tilting subcategory.
%\end{itemize}
\end{thm}

\begin{rem}
We only consider Frobenius extriangulated category in this theorem.  Since when  $\M$ is cluster tilting, it has to contain all the projectives and injectives, if there is a non-projective injective object $I$, $\Omega I$ is a project object in $\oB$ such that $\Hom_{\oB}(\Omega I,\overline \M)=0$, which means in such case $\overline \M$ can never be a tilting subcategory.  %In the opposite category $\B^{\op}$, $\R$ and $\M$ are still cluster tilting. Then $\overline \M$ is a tilting subcategory in $\oB^{\op}$, which implies that $\overline \M$ is also a cotilting subcategory in $\oB$.
\end{rem}

In the rest of this section, let $\B$ be a Frobenius extriangulated category. To show this theorem, we need several lemmas.

%\begin{lem}\label{imp}
%If $\M\subseteq \B$ is rigid, then $\overline \M$ is rigid in $\oB$.
%\end{lem}

%\begin{proof}
%Since $\M$ is rigid, it is also $\R$-rigid. By Theorem \ref{main1}, we have that $\overline \M$ is $\tau$-rigid.
%Combine with \cite[Lemma 5.2]{IJY}, we get that $\overline \M$ is rigid.
%\end{proof}

\begin{lem}\label{imp}
Let $X,Y\in \B_\R$. Then $\EE(X,Y)=0$ implies $\Ext^1_{\oB}(X,Y)=0$.
\end{lem}

\begin{proof}
We can assume that $X,Y$ are indecomposable. Let
$$(\blacklozenge) \quad 0\to Y\to Z\xrightarrow{\overline f} X\to 0 \text{ }$$
 be a short exact sequence in $\oB$. The morphism $f$ admits the following commutative diagram
$$\xymatrix{
R_1 \ar[r] \ar@{=}[d] &Y' \ar[r]^{g'} \ar[d]^{y} &Z \ar[d]^{f} \ar@{-->}[r] &\\
R_1 \ar[r] &R_0 \ar[r]^c &X \ar@{-->}[r] &
}
$$
where $R_1,R_0 \in \R$. Then we have an $\EE$-triangle $\xymatrix{Y' \ar[r]^-{\svecv{g'}{-y}} &Z\oplus R_0 \ar[r]^-{\svech{f}{c}} &X \ar@{-->}[r] &}$ which induces a short exact sequence
$$(\lozenge) \quad 0\to Y'\xrightarrow{\overline g'} Z\xrightarrow{\overline f} X\to 0$$
in $\oB$ which is isomorphic to $(\blacklozenge)$. Hence $Y'=Y\oplus Y_0$ where $Y_0\in \R$. Denote morphism $Y' \xrightarrow{\svecv{g'}{-y}} Z\oplus R_0$ by $Y\oplus Y_0 \xrightarrow{\left(\begin{smallmatrix}
g_1'&g_2'\\
-y_1&-y_2
\end{smallmatrix}\right)} Z\oplus R_0$, then $\overline g'=\overline g_1'$. Since $\EE(X,Y)=0$, we have a commutative diagram
$$\xymatrix{
Y\oplus Y_0 \ar[d]_{\svecv{1}{0}} \ar[rr]^-{\left(\begin{smallmatrix}
g_1'&g_2'\\
-y_1&-y_2
\end{smallmatrix}\right)} &&Z\oplus R_0 \ar@{.>}[dll]^-{\svech{z_1}{z_2}} \ar[r]^-{\svech{f}{c}} &X \ar@{-->}[r] &\\
Y
}
$$
which implies $\overline 1_Y=\overline {z_1g_1'}$. Then $(\lozenge)$ splits, so is $(\blacklozenge)$.\\
\end{proof}

For an object $X\in \B$, we can always get the following commutative diagram $(\heartsuit)$.
$$\xymatrix{
\Omega X \ar[r]^{a} \ar[d]^{a'} &\Omega R^1 \ar@{=}[r] \ar[d]^b &\Omega R^1\ar[d]^{p_1}\\
P_X \ar[r]^{b'} \ar[d] &\Omega R^2 \ar[r]^{p_2} \ar[d]^c &P\ar[r]^{q_2} \ar[d]^{q_1} &R^2 \ar@{=}[d] \ar@{-->}[r] & \\
X \ar@{=}[r] \ar@{-->}[d] &X \ar[r]^x \ar@{-->}[d] &R^1 \ar@{-->}[d] \ar[r]^d &R^2 \ar@{-->}[r] &\\
& & &
}$$
where $R^1,R^2\in \R$, $P,P_X\in \mathcal P$.  We denote by $\pd_{\oB}(X)$ the projective dimension of $X$ in $\oB$. Then we have the following result.

\begin{lem}\label{imp2}
Let $X\notin \R$ and $\add X\cap \R\subseteq \mathcal P$. Then
%$\oB$ is $1$-Gorenstein if and only if
%\begin{itemize}
%\item[(a)] %For any object
$\pd_{\oB}(X)\leq 1$ if and only if in the diagram $(\heartsuit)$, the morphism $a:\Omega X\to \Omega R^1$ factors through $\R$. %for the short exact sequence $\xymatrix{X \ar@{ >->}[r]^x &C^1 \ar@{->>}[r] &C^2}$ where $C^1,C^2\in \C$, morphism $\overline{\Omega x}=0$.
%there exists a commutative diagram
%$$\xymatrix@C=0.7cm@R0.7cm{
%\Omega X \ar[d]_{\Omega x} \ar[r]^p &P_X \ar[d] \ar[r]^q &X \ar[d]^x \ar@{-->}[r] &\\
%\Omega C \ar[r]_{p_1} &P \ar[r]_{q_1} &C^1 \ar@{-->}[r] &
%}
%$$
%where $x$ is a left $\C$-approximation and an inflation, $P,P_X\in \mathcal P$, the morphism $\overline {\Omega x}=0$;
%in the diagram $(\clubsuit)$, $\overline a=0$;
%\item[(b)] $\id_{\oB}(X)\leq 1$ if and only if %for the short exact sequence $\xymatrix{D_2 \ar@{ >->}[r] &D_1 \ar@{->>}[r]^y &Y}$
%there exists a commutative diagram
%$$\xymatrix@C=0.7cm@R0.7cm{
%D_1 \ar[d]^y \ar[r] &I \ar[r] \ar[d] &\Sigma D_1 \ar@{-->}[r] \ar[d]^{\Sigma y} &\\
%X \ar[r] &I^X \ar[r] &\Sigma X \ar@{-->}[r] &
%}
%$$
%where $y$ is a right $\D$-approximation and a deflation $I,I^X\in \mathcal I$, the morphism $\overline {\Sigma y}=0$.
%\end{itemize}
\end{lem}

\begin{proof}
%We only prove (a), (b) is by dual.

The diagram $(\heartsuit)$ induces an exact sequence $\Omega X \xrightarrow{\overline a} \Omega R^1 \xrightarrow{\overline b} \Omega R^2 \xrightarrow{\overline c} X\to 0$ in $\oB$. %Moreover, we have the following commutative diagram
%$$\xymatrix{
%\Omega X \ar[d]_{-a} \ar[r]^p &P_X \ar[d]^{-p_2b'} \ar[r]^q &X \ar[d]^x \ar@{-->}[r] &\\
%\Omega C^1 \ar[r]_{p_1} &P \ar[r]_{q_1} &C^1 \ar@{-->}[r] &.
%}
%$$
%It is not hard to check that $-\underline a=\overline \Omega x$, then $-\overline a=\overline {\Omega x}$.
%If $-\overline a=\overline {\Omega x}=0$,
If $a$ factors through $\R$, $X$ admits a short exact sequence $0\to \Omega R^1 \to \Omega R^2 \to X \to 0$,
%If $\overline b=0$, any indecomposable object $Y\in  (\Omega \C)_{\Sigma \D}\cap \h_\K$ admits a short exact sequence $0\to \Y \to \Sigma D_2 \to \Sigma D_1 \to 0$ in $\oB$.\\
hence $\pd_{\oB}(X)\leq 1$.
\smallskip

Now we prove the ``only if" part. We can assume that $X$ is indecomposable in $\oB$.

\smallskip

Since $X\notin \R$, we have $\overline c\neq 0$. If $\overline b=0$, then $X$ is a direct summand of $\Omega R^2$. Hence $X\in \Omega \R$ and we can take $R^1\in\mathcal P$. Hence $\Omega R^1\in \mathcal P$ and $\overline a=0$. Now let $\overline b\neq 0$. Then we have the following exact sequence:
$$\xymatrix@C=0.5cm@R0.4cm{\Omega X \ar[rr]^{\overline a} \ar@{->>}[dr]_{\overline {r_4}} &&\Omega R^1 \ar[rr]^{\overline b} \ar@{->>}[dr]_{\overline {r_2}} &&\Omega R^2 \ar[r]^-{\overline c} &X \ar[r] &0\\
&S_2\ar@{ >->}[ur]_{\overline {r_3}}  &&S_1 \ar@{ >->}[ur]_{\overline {r_1}}}
$$
where $\Omega R^1 \xrightarrow{\overline {r_2}} S_1\xrightarrow{\overline {r_1}} \Omega R^2$ is an epic-monic factorization of $\overline b$ and $\Omega X \xrightarrow{\overline {r_4}} S_2\xrightarrow{\overline {r_3}} \Omega R^1$ is an epic-monic factorization of $\overline a$. Since $\pd_{\oB}(X)\leq 1$, we have $S_1\in \overline {\Omega \R}$, hence we get a split short exact sequence $0\to S_2\xrightarrow{\overline {r_3}} \Omega R^1 \xrightarrow{\overline {r_2}} S_1\to 0$ which implies $S_2\in \overline{\Omega \R}$. Thus $S_2$ is a direct summand of $\Omega X$. But $\Omega X$ is indecomposable in $\oB$, if $S_2\neq 0$ in $\oB$, then $\Omega X\simeq S_2\in \overline {\Omega \R}$, which means $X\in \R$, a contradiction. Hence $S_2\in \R$ and $\overline a=0$.
\end{proof}

\begin{lem}\label{ses}
Let $X\notin \R$ such that $\add X\cap \R\subseteq \mathcal P$ and $\pd_{\oB}(X)\leq 1$. Let $Y\notin \R$. If we have an $\EE$-triangle $$\xymatrix{Y \ar[r]^g &Z \ar[r]^f &X \ar@{-->}[r] &}$$
such that $\overline f$ is an epimorphism in $\oB$, then this $\EE$-triangle induces a short exact sequence $$0\to Y\xrightarrow{~\overline g~}  Z\xrightarrow{~\overline f~} X\to 0.$$
\end{lem}

\begin{proof}
%To show the "only if" part, we need neither the condition "$\pd_{\oB}(X)\leq 1$" nor "$\id_{\oB}(X)\leq 1$".\\
By hypothesis, we already have an exact sequence $ Y\xrightarrow{\overline g} Z\xrightarrow{\overline f}  X\to 0$. We only need to check that $\overline g$ is a monomorphism.

$X$ admits the following commutative diagram
$$\xymatrix{
\Omega R^1 \ar@{=}[r] \ar[d]^b &\Omega R^1\ar[d]^{p_1}\\
\Omega R^2 \ar[r]^{p_2} \ar[d]^c &P\ar[r]^{q_2} \ar[d]^{q_1} &R^2 \ar@{=}[d] \ar@{-->}[r] &\\
X \ar[r]^x \ar@{-->}[d] &R^1 \ar@{-->}[d] \ar[r]^d &R^2 \ar@{-->}[r] &.\\
& &
}$$
Since $\overline f$ is an epimorphism and $\Omega R^2$ is projective in $\oB$, there is a morphism $z:\Omega R^2\to Z$ such that $\overline {fz}=\overline c$. Then there is a morphism $q':P\to X$ such that $fz-c=q'p_2$. But $P$ is projective, hence there is a morphism $q'':P\to Z$ such that $fq''=q'$. It follows that $f(z-q''p_2)=c$ and  we have the following commutative diagram
$$
\xymatrix{
\Omega X \ar[r]^{a'} \ar[d]_{a}  &P_X \ar[r] \ar[d] &X \ar@{=}[d] \ar@{-->}[r] &\\
\Omega R^1 \ar[r] \ar[d]_{d'} &\Omega R^2 \ar[r]^c \ar[d]^{z-q''p_2} &X \ar@{=}[d] \ar@{-->}[r] &\\
Y \ar[r]_g &Z \ar[r]_f &X \ar@{-->}[r] &
}
$$
where $a$ factors through $\R$. Then we have an $\EE$-triangle $\xymatrix{\Omega X \ar[r]^-{\svecv{d'a}{a'}} &Y\oplus P_X \ar[r]^-{\svech{g}{*}} &Z \ar@{-->}[r] &}$ which induces an exact sequence $\Omega X\xrightarrow{0} Y\xrightarrow{~\overline g~} Z$. Hence $\overline g$ is a monomorphism.
\end{proof}

Now we are ready to prove Theorem \ref{main4}.

\begin{proof}
Let $\M$ be a cluster tilting subcategory such that $\M\cap \R=\mathcal P$. Any indecomposable object $\Omega R\in \Omega \R$ such that $\Omega R\notin \M$ admits a commutative diagram
$$\xymatrix{
\Omega R \ar[r]^f \ar@{=}[d] &M_1 \ar[d] \ar[r]^g &M_2 \ar[d] \ar@{-->}[r] &\\
\Omega R \ar[r] &P \ar[r] &R \ar@{-->}[r] &
}$$
where $M_1,M_2\in \M$ and $P\in \mathcal P$.  We get that $\overline g$ is an epimorphism, hence by Lemma \ref{ses}, we have a short exact sequence $0\to \Omega R\to M_1\to M_2\to 0$. Since $\EE(\M,\M)=0$ implies $\Ext^1_{\oB}(\M,\M)=0$ by Lemma \ref{imp}, we get that $\overline \M$ is a tilting subcategory. %Dually one can show that $\overline \M$ is also a cotilting subcategory.
%(b) Now Let $\M\subseteq \B_\R$ be a contravariantly finite subcategory such that $\overline \M$ is a tilting subcategory. We show that $\M$ is a cluster tilting subcategory. By the result in \cite{KZ}, it is enough to show that for any object $X\in \B_\R$, $\Hom_{\B}(\M,X[1])=0$ implies $X\in \M$.\\
%Let $X$ be an indecomposable object such that $\Hom_{\B}(\M,X[1])=0$
\end{proof}

%For a morphism $x:A\to X$ (or $x:X\to B$), let $[x](A,B)$ be the subgroup of $\Hom_{\B}(A,B)$ consisting of morphisms which factor through $x$. For another morphism $x':A\to X$ (or $x':X\to B$), let $[x,x'](A,B)=[x](A,B)\cap [x'](A,B)$. We show the following lemma, which will be used in the next section.

%\subsection{When $\oB$ is hereditary}

%Now we assume that $\oB$ is hereditary. We introduce some notions first.

\section{Applications}

In this section, let $\B$ be a triangulated category with shift functor $\Sigma$. %For any subcategory $\D\subseteq \B$, let $\D^{\bot_1}=\{X\in \B \text{ }|\text{ }\Ext^1_{\oB}(\D,X) \}$

We show the following proposition, which is a generalization of \cite[Theorem 3.5]{HJ} and \cite[Theorem 6.6]{B}.

\begin{prop}\label{TS}
Let $\B$ be a triangulated category with Serre functor $\mathbb{S}$. Let $\X \subseteq \h$ be a maximal $\R$-rigid subcategory such that $\mathbb{S}\X=\Sigma^2 \X$ and $\X\neq \R$. If $\X$ is contravariantly finite, then $\X$ is cluster tilting.
\end{prop}

\begin{proof}
By \cite[Proposition 3.4]{YZ}, when $\mathbb{S}\X=\Sigma^2\X$, $\X$ is $\R$-rigid implies that $\X$ is rigid. To show that $\X$ is cluster tilting, it is enough to prove that if $\Hom_{\B}(\X,\Sigma M)=0$, then $M\in \X$.\\
Since $\mathbb{S}\X=\Sigma^2\X$, $\Hom_{\B}(\X,\Sigma M)=0$ implies that $\Hom_{\B}(M,\Sigma \X)=0$. Then we have $[\R](M,\Sigma \X)=0=[\overline \R](M,\Sigma \X)$ and $[\R](\X,\Sigma M)=0=[\overline \R](\X,\Sigma M)$. Hence by Proposition \ref{maximal}, we have $M\in \X$.
\end{proof}

When $\B$ is a 2-Calabi-Yau triangulated category, any subcategory in $\h$ is maximal $\R$-rigid if and only if it is maximal rigid. Then we have the following corollary, which generalizes \cite[Theorem 2.6]{ZhZ}.

\begin{cor}\label{2CY}
Let $\B$ be a 2-Calabi-Yau triangulated category and $\R$ be a contravariantly finite rigid subcategory. Then any contravariantly finite maximal rigid subcategory $\X \subseteq \h$ such that $\X\neq \R$ is cluster tilting.
\end{cor}

We show the following proposition, which generalizes \cite[Corollary 6.9]{B}.

\begin{prop}\label{completion3}
Let $\B$ be a triangulated category and $\R$ be cluster tilting. Let $\U\subseteq \B_\R$ be a contravariantly finite rigid subcategory such that $\pd_{\oB}(U)\leq 1$, $\forall U\in \U$. Denote $\{M\in \oB \text{ }|\text{ } \Ext^1_{\oB}(\overline \U,M)=0 \text{ and } \Ext^1_{\oB}(M,N)=0 \text{ if } \Ext^1_{\oB}(\overline \U,N)=0\}$ by $\mathrm{P}(\overline \U^{\bot_1})$. Then $\mathrm{P}(\overline \U^{\bot_1})$ is a tilting subcategory of $\oB$.
\end{prop}

\begin{proof}
Let $R\in \R$. If $\Hom_{\B}(\U,R)=0$, we have $\Ext^1_{\B}(\U,\Sigma^{-1}R)=0$. By Lemma \ref{imp}, we obtain $\Ext^1_{\oB}(\overline \U,\Sigma^{-1}R)=0$. Since $\Sigma^{-1}R$ is projective, by the definition we have $\Sigma^{-1}R\in \mathrm{P}(\overline \U^{\bot_1})$.

%Let $R\in \R$ and $\Hom_{\B}(R,\Sigma \U)\neq 0$. Then we have a triangle $\Sigma^{-1} R\xrightarrow{u} U \xrightarrow{v} V\to R$ where $ u$ is a left $\U$-approximation and $\overline u\neq 0$. Since $\overline \U\subseteq \overline \Y$ and $\overline \Y$ is covariantly finite, any $R\in \R$ such that $\Hom_{\B}(R,\Sigma \U)\neq 0$ admits a non-zero left $\overline Y$-approximation.\\
If $\Hom_{\B}(\U,R)\neq 0$, $\R$ admits a triangle $\Sigma^{-1} R\xrightarrow{r} V' \xrightarrow{v'} U'\xrightarrow{u'} R$ where $u'\neq 0$ is a right $\U$-approximation. Then $\Ext^1_{\B}(\U,V')=0$ and hence $\Ext^1_{\oB}(\overline \U,V')=0$. Since $\pd_{\oB}(\U)\neq 1$, by Lemma \ref{ses}, we have a short exact sequence $0\to \Sigma^{-1} R\xrightarrow{\overline r} V' \xrightarrow{\overline v} U'\to 0$ in $\oB$. For any object $N$ such that $\Ext^1_{\oB}(\overline \U,N)=0$, we have the following exact sequence
$$0=\Ext^1_{\oB}(U',N)\to \Ext^1_{\oB}(V',N)\to \Ext^1_{\oB}(\Sigma^{-1}R,N)=0$$
which implies that $V'\in \mathrm{P}(\overline \U^{\bot_1})$. Since $\mathrm{P}(\overline \U^{\bot_1})$ is rigid and $\oB$ is hereditary, by the definition $\mathrm{P}(\overline \U^{\bot_1})$ is a tilting subcategory of $\oB$.
%Moreover, we have $\overline r$ is a left $\overline \Y$-approximation of $\Sigma^{-1} R$. Since $\overline \Y$ is rigid and $\oB$ is hereditary, $\overline \Y$ is $\tau$-rigid. Hence by definition $\overline \Y$ is
\end{proof}

\section{Examples}

In this section, we give two examples to explain our main results.

\begin{exm}\label{ex1}
Let $\Lambda$ be the $k$-algebra given by the quiver
$$
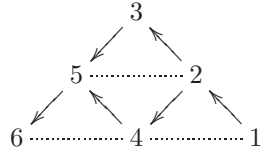
\xymatrix@C=0.4cm@R0.4cm{
&&3 \ar[dl]\\
&5 \ar[dl] \ar@{.}[rr] &&2 \ar[dl] \ar[ul]\\
6 \ar@{.}[rr] &&4 \ar[ul] \ar@{.}[rr] &&1 \ar[ul]}$$
with mesh relations. The Auslander-Reiten quiver of $\B:=\mod\Lambda$ is given by
$$
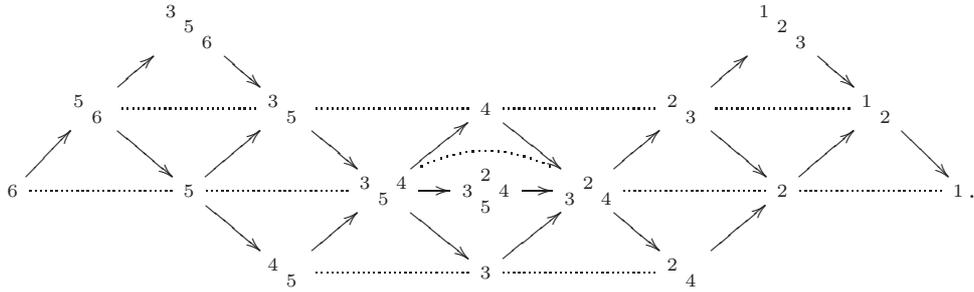
\xymatrix@C=0.4cm@R0.4cm{
&&{\begin{smallmatrix}
3&&\\
&5&\\
&&6
\end{smallmatrix}} \ar[dr] &&&&&&{\begin{smallmatrix}
1&&\\
&2&\\
&&3
\end{smallmatrix}} \ar[dr]\\
&{\begin{smallmatrix}
5&&\\
&6&
\end{smallmatrix}} \ar[ur] \ar@{.}[rr] \ar[dr] &&{\begin{smallmatrix}
3&&\\
&5&
\end{smallmatrix}} \ar@{.}[rr] \ar[dr] &&{\begin{smallmatrix}
4
\end{smallmatrix}} \ar@{.}[rr] \ar[dr] &&{\begin{smallmatrix}
2&&\\
&3&
\end{smallmatrix}} \ar[ur] \ar@{.}[rr] \ar[dr] &&{\begin{smallmatrix}
1&&\\
&2&
\end{smallmatrix}} \ar[dr]\\
{\begin{smallmatrix}
6
\end{smallmatrix}} \ar[ur] \ar@{.}[rr] &&{\begin{smallmatrix}
5
\end{smallmatrix}} \ar[ur] \ar@{.}[rr] \ar[dr] &&{\begin{smallmatrix}
3&&4\\
&5&
\end{smallmatrix}} \ar[ur] \ar[r] \ar[dr] \ar@{.}@/^15pt/[rr] &{\begin{smallmatrix}
&2&\\
3&&4\\
&5&
\end{smallmatrix}} \ar[r] &{\begin{smallmatrix}
&2&\\
3&&4
\end{smallmatrix}} \ar[ur] \ar@{.}[rr] \ar[dr] &&{\begin{smallmatrix}
2
\end{smallmatrix}} \ar[ur] \ar@{.}[rr] &&{\begin{smallmatrix}
1
\end{smallmatrix}}.\\
&&&{\begin{smallmatrix}
4&&\\
&5&
\end{smallmatrix}} \ar[ur] \ar@{.}[rr] &&{\begin{smallmatrix}
3
\end{smallmatrix}} \ar[ur] \ar@{.}[rr] &&{\begin{smallmatrix}
2&&\\
&4&
\end{smallmatrix}} \ar[ur]
}$$
We denote by ``~$\bullet$" in the Auslander-Reiten quiver the indecomposable objects belong to a subcategory and by ``~$\circ$'' the indecomposable objects do not belong to it.
$$
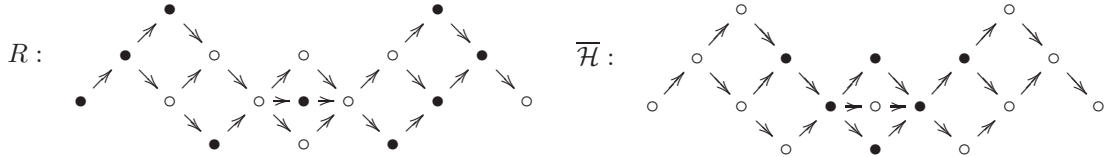
\xymatrix@C=0.2cm@R0.2cm{
&&&\bullet \ar[dr] &&&&&&\bullet \ar[dr]\\
{R:} &&\bullet \ar[ur]  \ar[dr] &&\circ  \ar[dr] &&\circ  \ar[dr] &&\circ  \ar[ur]  \ar[dr] &&\bullet \ar[dr]\\
&\bullet \ar[ur]  &&\circ \ar[ur]  \ar[dr] &&\circ \ar[ur] \ar[r] \ar[dr] &\bullet \ar[r] &\circ \ar[ur] \ar[dr] &&\bullet \ar[ur] &&\circ\\
&&&&\bullet \ar[ur] &&\circ \ar[ur] &&\bullet \ar[ur]
\\} \quad
\xymatrix@C=0.2cm@R0.2cm{
&&&\circ \ar[dr] &&&&&&\circ \ar[dr]\\
{\overline \h:} &&\circ \ar[ur]  \ar[dr] &&\bullet  \ar[dr] &&\bullet  \ar[dr] &&\bullet  \ar[ur]  \ar[dr] &&\circ \ar[dr]\\
&\circ \ar[ur]  &&\circ \ar[ur]  \ar[dr] &&\bullet \ar[ur] \ar[r] \ar[dr] &\circ \ar[r] &\bullet \ar[ur] \ar[dr] &&\circ \ar[ur] &&\circ\\
&&&&\circ \ar[ur] &&\bullet \ar[ur] &&\circ \ar[ur]
}$$
We give some maximal $\R$-rigid objects of $\B$:
\vspace{1mm}
$${\begin{smallmatrix}
\ &3&\
\end{smallmatrix}} \oplus
\begin{smallmatrix}
3&&4\ \\
&5&
\end{smallmatrix} \oplus {\begin{smallmatrix}
\ &4&\
\end{smallmatrix}}\oplus \Lambda, \quad {\begin{smallmatrix}
\ &3&\
\end{smallmatrix}} \oplus
\begin{smallmatrix}
&2&\ \\
3&&4
\end{smallmatrix} \oplus {\begin{smallmatrix}
\ &4&\
\end{smallmatrix}}\oplus \Lambda,\quad{{\begin{smallmatrix}
\ &3&\
\end{smallmatrix}} \oplus
\begin{smallmatrix}
&2&\ \\
3&&4
\end{smallmatrix}} \oplus {\begin{smallmatrix}
2&&\\
&3&
\end{smallmatrix}} \oplus \Lambda, \quad {{\begin{smallmatrix}
\ &4&\
\end{smallmatrix}} \oplus
\begin{smallmatrix}
&2&\ \\
3&&4
\end{smallmatrix}}\oplus \Lambda \oplus {\begin{smallmatrix}
2&&\\
&4&
\end{smallmatrix}}.
$$
\vspace{1mm}
By Theorem \ref{main2}, they are the liftings of the following support $\tau$-tilting objects in $\overline \h$:
\vspace{1mm}$${\begin{smallmatrix}
\ &3&\
\end{smallmatrix}} \oplus
\begin{smallmatrix}
3&&4\ \\
&5&
\end{smallmatrix} \oplus {\begin{smallmatrix}
\ &4&\
\end{smallmatrix}},\quad{\begin{smallmatrix}
\ &3&\
\end{smallmatrix}} \oplus
\begin{smallmatrix}
&2&\ \\
3&&4
\end{smallmatrix} \oplus {\begin{smallmatrix}
\ &4&\
\end{smallmatrix}},\quad {{\begin{smallmatrix}
\ &3&\
\end{smallmatrix}} \oplus
\begin{smallmatrix}
&2&\ \\
3&&4
\end{smallmatrix}} \oplus {\begin{smallmatrix}
2&&\\
&3&
\end{smallmatrix}}, \quad {{\begin{smallmatrix}
\ &4&\
\end{smallmatrix}} \oplus
\begin{smallmatrix}
&2&\ \\
3&&4
\end{smallmatrix}}.
$$\\
Moreover,
${\begin{smallmatrix}
\ &3&\
\end{smallmatrix}} \oplus
\begin{smallmatrix}
3&&4\ \\
&5&
\end{smallmatrix} \oplus {\begin{smallmatrix}
\ &4&\
\end{smallmatrix}},\quad {\begin{smallmatrix}
\ &3&\
\end{smallmatrix}} \oplus
\begin{smallmatrix}
&2&\ \\
3&&4
\end{smallmatrix} \oplus {\begin{smallmatrix}
\ &4&\
\end{smallmatrix}},\quad {{\begin{smallmatrix}
\ &3&\
\end{smallmatrix}} \oplus
\begin{smallmatrix}
&2&\ \\
3&&4
\end{smallmatrix}} \oplus {\begin{smallmatrix}
2&&\\
&3&
\end{smallmatrix}}$
are tilting objects, and
$({{\begin{smallmatrix}
\ &4&\
\end{smallmatrix}} \oplus
\begin{smallmatrix}
&2&\ \\
3&&4
\end{smallmatrix}}, {\begin{smallmatrix}
3&&\\
&5&
\end{smallmatrix}})
$
is a support $\tau$-tilting pair in $\overline \h$.\\
\end{exm}

\begin{exm}
Let $Q\colon 1\to 2\to 3\to 4$ be the quiver of type $A_4$ and $\B:=\mathrm{D}^b(kQ)$ the bounded derived category of $kQ$ whose Auslander-Reiten quiver is the following:
$$
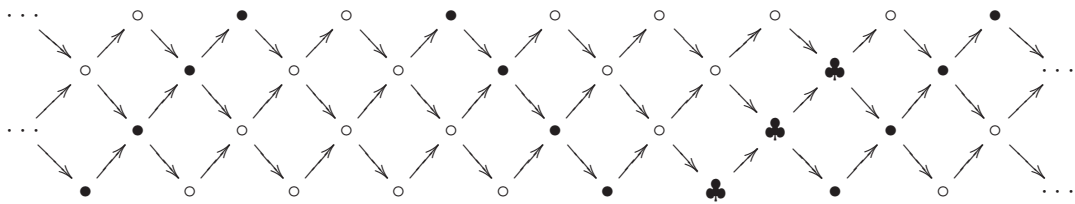
\xymatrix@C=0.3cm@R0.3cm{
\cdot\cdot\cdot \ar[dr] &&\circ \ar[dr] &&\bullet\ar[dr] &&\circ \ar[dr] &&\bullet \ar[dr] &&\circ \ar[dr] &&\circ \ar[dr] &&\circ \ar[dr] &&\circ \ar[dr] &&\bullet \ar[dr]\\
&\circ \ar[ur] \ar[dr] &&\bullet \ar[ur] \ar[dr] &&\circ \ar[ur] \ar[dr] &&\circ \ar[ur]  \ar[dr] &&\bullet \ar[ur] \ar[dr]  &&\circ \ar[ur] \ar[dr] &&\circ \ar[ur]  \ar[dr] &&\clubsuit \ar[ur] \ar[dr] &&\bullet\ar[ur] \ar[dr] &&\cdot\cdot\cdot\\
\cdot\cdot\cdot \ar[dr] \ar[ur] &&\bullet \ar[ur] \ar[dr] &&\circ \ar[ur] \ar[dr] &&\circ \ar[ur] \ar[dr]  &&\circ \ar[ur] \ar[dr] &&\bullet \ar[ur] \ar[dr] &&\circ \ar[ur] \ar[dr] &&\clubsuit \ar[ur] \ar[dr] &&\bullet \ar[ur] \ar[dr] &&\circ \ar[ur] \ar[dr]\\
&\bullet \ar[ur] &&\circ \ar[ur] &&\circ \ar[ur] &&\circ \ar[ur] &&\circ \ar[ur] &&\bullet \ar[ur] &&\clubsuit \ar[ur] &&\bullet \ar[ur] &&\circ \ar[ur] &&\cdot\cdot\cdot
}
$$
Let $\R$ be the subcategory whose indecomposable objects are marked by bullets
here (these bullets will appear periodically on both side). We know that $\R$ is a cluster tilting subcategory of $\B$. Consider three objects denoted by $\clubsuit$, let $\U\subseteq \B$ be the subcategory generated by them. Then $\U$ is a contravariantly finite rigid subcategory. By Proposition \ref{completion}, $\mathrm{P}(\Fac \overline \U)$, which is generated by the following four objects indicated by $\heartsuit$:
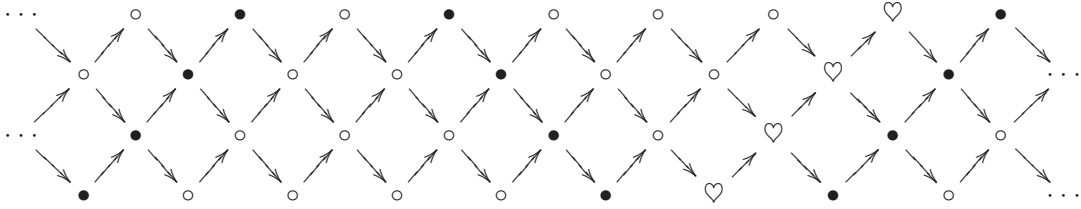
$$\xymatrix@C=0.3cm@R0.3cm{
\cdot\cdot\cdot \ar[dr] &&\circ \ar[dr] &&\bullet\ar[dr] &&\circ \ar[dr] &&\bullet \ar[dr] &&\circ \ar[dr] &&\circ \ar[dr] &&\circ \ar[dr] &&\heartsuit \ar[dr] &&\bullet \ar[dr]\\
&\circ \ar[ur] \ar[dr] &&\bullet \ar[ur] \ar[dr] &&\circ \ar[ur] \ar[dr] &&\circ \ar[ur]  \ar[dr] &&\bullet \ar[ur] \ar[dr]  &&\circ \ar[ur] \ar[dr] &&\circ \ar[ur]  \ar[dr] &&\heartsuit \ar[ur] \ar[dr] &&\bullet\ar[ur] \ar[dr] &&\cdot\cdot\cdot\\
\cdot\cdot\cdot \ar[dr] \ar[ur] &&\bullet \ar[ur] \ar[dr] &&\circ \ar[ur] \ar[dr] &&\circ \ar[ur] \ar[dr]  &&\circ \ar[ur] \ar[dr] &&\bullet \ar[ur] \ar[dr] &&\circ \ar[ur] \ar[dr] &&\heartsuit \ar[ur] \ar[dr] &&\bullet \ar[ur] \ar[dr] &&\circ \ar[ur] \ar[dr]\\
&\bullet \ar[ur] &&\circ \ar[ur] &&\circ \ar[ur] &&\circ \ar[ur] &&\circ \ar[ur] &&\bullet \ar[ur] &&\heartsuit \ar[ur] &&\bullet \ar[ur] &&\circ \ar[ur] &&\cdot\cdot\cdot
}
$$
Moreover, by Proposition \ref{completion3}, $\overline \U$ is contained in a tilting subcategory $\mathrm{P}(\overline \U^{\bot_1})=:\{M\in \oB \text{ }|\text{ } \Ext^1_{\oB}(\overline \U,M)=0 \text{ and } \Ext^1_{\oB}(M,N)=0 \text{ if } \Ext^1_{\oB}(\overline \U,N)=0\}$, which is indicated by $\spadesuit$
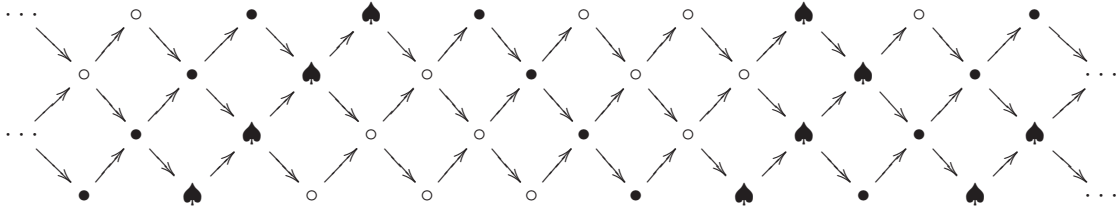
$$\xymatrix@C=0.3cm@R0.3cm{
\cdot\cdot\cdot \ar[dr] &&\circ \ar[dr] &&\bullet\ar[dr] &&\spadesuit \ar[dr] &&\bullet \ar[dr] &&\circ \ar[dr] &&\circ \ar[dr] &&\spadesuit \ar[dr] &&\circ \ar[dr] &&\bullet \ar[dr]\\
&\circ \ar[ur] \ar[dr] &&\bullet \ar[ur] \ar[dr] &&\spadesuit \ar[ur] \ar[dr] &&\circ \ar[ur]  \ar[dr] &&\bullet \ar[ur] \ar[dr]  &&\circ \ar[ur] \ar[dr] &&\circ \ar[ur]  \ar[dr] &&\spadesuit \ar[ur] \ar[dr] &&\bullet\ar[ur] \ar[dr] &&\cdot\cdot\cdot\\
\cdot\cdot\cdot \ar[dr] \ar[ur] &&\bullet \ar[ur] \ar[dr] &&\spadesuit \ar[ur] \ar[dr] &&\circ \ar[ur] \ar[dr]  &&\circ \ar[ur] \ar[dr] &&\bullet \ar[ur] \ar[dr] &&\circ \ar[ur] \ar[dr] &&\spadesuit \ar[ur] \ar[dr] &&\bullet \ar[ur] \ar[dr] &&\spadesuit \ar[ur] \ar[dr]\\
&\bullet \ar[ur] &&\spadesuit \ar[ur] &&\circ \ar[ur] &&\circ \ar[ur] &&\circ \ar[ur] &&\bullet \ar[ur] &&\spadesuit \ar[ur] &&\bullet \ar[ur] &&\spadesuit \ar[ur] &&\cdot\cdot\cdot
}
$$
$\mathrm{P}(\overline \U^{\bot_1})$ has infinite many objects. It is not cluster tilting (it is not rigid). This also shows that the Serre functor condition in Proposition \ref{TS} is necessary. %We can also get
%And we have one short exact sequence $0\to \spadesuit \to \clubsuit \to \heartsuit\to 0$ as mentioned in Proposition \ref{main3}.
\end{exm}

\end{document}